\documentclass{amsart}

\usepackage[T1]{fontenc}

\usepackage{graphicx}

\usepackage{hyperref}
\usepackage[latin1]{inputenc}
\usepackage{amssymb}
\usepackage{amsmath}
\usepackage{esint}
\usepackage{leftidx}

\newcommand{\rn}{{\mathbb{R}^n}}
\newcommand{\phii}{\varphi}
\newcommand{\e}{\varepsilon}
\newcommand{\ele}{L^2(\Omega)}
\newcommand{\Fu}{\mathcal{F}}
\newcommand{\pp}{\overline{\partial}}
\newcommand{\rl}{\leftidx{^{RL}}\partial_t^\alpha}
\newcommand{\ppa}{\leftidx{^C}\partial_t^\alpha}
\newcommand{\w}{\omega}
\newcommand{\W}{\Omega}
\newcommand{\eps}{\varepsilon}
\newcommand{\dt}{\partial_{\tau}}
\newcommand{\dtd}{\overline{\partial_{\tau}}}

\newcommand{\hsT}{\widetilde{H}^s (\W)}
\newcommand{\hmsT}{H^{-s} (\W)}
\newcommand{\wtilde}{\tilde{\w}}
\newcommand{\fmonio}{g}

\newcommand{\V}{\mathbb{V}^q}

\usepackage{color}

\newcommand\restr[2]{{
  \left.\kern-\nulldelimiterspace 
#1 
  \vphantom{\big|} 
  \right|_{#2} 
  }}

\newcommand{\subf}[2]{%
  {\small\begin{tabular}[t]{@{}c@{}}
  #1\\#2
  \end{tabular}}%
}

\def\R{{\mathbb {R}}}
\def\N{{\mathbb {N}}}

\newtheorem{theorem}{Theorem}[section]
\newtheorem{lemma}[theorem]{Lemma}

\theoremstyle{remark}
\newtheorem{remark}[theorem]{Remark}

\theoremstyle{definition}

\numberwithin{equation}{section}

\title[Numerical approximations for a fully fractional
Allen-Cahn equation]{Numerical approximations for a fully fractional Allen-Cahn equation}

\author[G. Acosta and F. M. Bersetche ]{Gabriel Acosta and Francisco M. Bersetche}

\address[G. Acosta and F. M. Bersetche ]{IMAS - CONICET and Departamento de Matem\'a\-tica, FCEyN - Universidad de Buenos Aires, Ciudad Universitaria, Pabell\'on I  (1428) Buenos Aires, Argentina.}

\email[G. Acosta]{gacosta@dm.uba.ar}

\urladdr[G. Acosta]{http://mate.dm.uba.ar/~gacosta/}

\email[F. M. Bersetche]{fbersetche@dm.uba.ar}


\begin{document}

\begin{abstract}

A finite element scheme for an entirely fractional Allen-Cahn equation with non-smooth initial data is introduced and analyzed. In the proposed nonlocal model, the Caputo fractional in-time derivative and the fractional Laplacian replace the standard local operators. Piecewise linear finite elements and convolution quadratures are the basic 
tools involved in the presented numerical method.  Error analysis and implementation issues  are addressed together with the needed results of regularity for the continuous model.  Also, the asymptotic behavior of solutions, for a vanishing fractional parameter and usual derivative in time, is discussed within the framework of the $\Gamma$-convergence theory.

\end{abstract}

\maketitle

\section{Introduction}
Several physical and social phenomena  have shown to be efficiently described by means of  nonlocal models. From anomalous diffusion to peridynamics and from image processing to finance,  a rich collection of applications pervades 
the recent scientific literature, showing the relevance and versatility of this kind of models. In particular, many classical problems have been extended from the local to the nonlocal context in order to capture behaviors that are beyond the modeling capabilities of differential operators.  
A fact that, in turn, have nurtured the interest in their mathematical foundations as well the corresponding development of numerical methods. 
Basic examples of this kind of models can be elaborated by considering  the fractional Laplace operator,
\begin{equation}
	(-\Delta)^s u (x) = C(n,s) \mbox{ P.V.} \int_\rn \frac{u(x)-u(y)}{|x-y|^{n+2s}} \, dy,
	\label{eq:fraccionario}
\end{equation}
where $0<s<1$ and $ C(n,s) = \frac{2^{2s} s \Gamma(s+\frac{n}{2})}{\pi^{n/2} \Gamma(1-s)} $ is a normalization constant. 
From the probabilistic point of view, $(-\Delta)^s $ corresponds with the infinitesimal generator of a stable L\'evy process and can be shown, by means of the Fourier transform \cite{Hitchhikers}, that  the standard laplacian and the identity operator can be recovered  when $s\to 1^{-}$ or $s\to 0^{+}$, respectively.  Thanks to this, the fractional laplacian can be used to model a wide range of anomalous diffusion processes, where  particles are allowed to perform arbitrarily long jumps. 

On the other hand, even though \eqref{eq:fraccionario} makes perfect sense  for $n=1$,  \emph{lateral} versions are necessary for dealing with the time variable. In such a case, the so-called   Caputo and Riemman-Liouville derivatives of order $0<\alpha\le 1$, given respectively  by  
\begin{equation} \label{eq:caputo}
	\ppa u (x,t) =
	\left\lbrace
	\begin{array}{rl}
		\frac{1}{\Gamma(1-\alpha)} \int_0^t \frac{1}{(t-r)^{\alpha}} \frac{\partial u}{\partial r} (x,r) \, dr & \mbox{ if } 0 < \alpha < 1, \\
		\frac{\partial u}{\partial t} u(x,t) &  \mbox{ if } \alpha = 1,
	\end{array}
	\right.
\end{equation}
and   
\begin{equation*} \label{eq:RL}
	\rl u (x,t) =
	\left\lbrace
	\begin{array}{rl}
		\frac{1}{\Gamma(1-\alpha)}  \frac{\partial  }{\partial t} \int_0^t \frac{1}{(t-r)^{\alpha}} u (x,r) \, dr & \mbox{ if } 0 < \alpha < 1, \\
		\frac{\partial u}{\partial t} u(x,t) &  \mbox{ if } \alpha = 1,
	\end{array}
	\right.
\end{equation*}
are widely used in applications. 

Our aim, in this paper, is to extended the classical Allen-Cahn equation
\begin{equation}
	\partial_t u -  \eps^2\Delta u   = f(u)  \text{ in } \Omega \times (0,T), 
	\label{eq:AC_classic}
\end{equation}
where $f(u) = u - u^3$ and $\W \subset \rn$ is a domain with smooth enough boundary, to the nonlocal setting.  Originally introduced to model the motion of phase boundaries in crystalline solids \cite{AC_79},   the unknown function $u$ represents the density of the components, describing full concentration of one of them where $u = 1$ (or $-1$). Remarkably, the original formulation of the phase-field models \cite{CH_58} contemplates nonlocal interactions, and have been subsequently simplified and approximated by local models. 

In this way, we focus on the following problem, 
\begin{equation}
	\left\lbrace
	\begin{array}{rl}      
		\ppa u +  \eps^2(-\Delta)^s u  & = f(u)  \text{ in } \Omega \times (0,T), \\
		u(0)  & = v  \text{ in } \Omega, \\
		u & =  0  \text{ in }  \Omega^c\times[0,T],
	\end{array}
	\right.
	\label{eq:AC}
\end{equation}
where $v$ belongs to a suitable fractional Sobolev space. Our model \eqref{eq:AC} is based on the Caputo's version, due to its compliance with standard formulations based on initial conditions \cite{Diethelm}.

Several numerical techniques have been recently developed for space and time non-local versions of equation \eqref{eq:AC_classic}, most of them based on finite differences or spectral methods  \cite{akagi,HPH,HTY,LWY,hongwang,SXKE}. Also, numerical methods have been studied for nonlocal versions of related phase separation models, like the Cahn-Hilliard equation  \cite{ains,ains2}. In  \cite{gal17}, a rigorous analysis of a general form of problem \eqref{eq:AC} is presented, providing existence and regularity results for a large class of operators, including the fractional Laplacian \eqref{eq:fraccionario} considered here. A similar analysis is also presented in \cite{tesisneto}.

The article has been organized in the following way. In  Section \ref{sec:AC_eq}, a theoretical treatment of a modified version of problem \eqref{eq:AC} with non-smooth initial datum, including existence, uniqueness and regularity results, is presented.  We focus on certain specific kind of regularity results that are  tailored to suit the analysis of our numerical method. In Section \ref{sec:ns} the numerical scheme, based on  Finite Elements for the spatial discretization and convolution quadrature rules for the time variable is presented and the error estimation is treated in Section \ref{sec:error} . Some arguments, developed in  Section \ref{sec:analisisAC}, show that our analysis can be extended to the model problem   \eqref{eq:AC}. A complementary result, inspired by the behavior of solutions obtained in our numerical simulations, is given in Section \ref{sec:asymptotic} where we study the asymptotic behavior of solutions of \eqref{eq:AC}, with $\alpha = 1$, when the fractional parameter $s \to 0^+$. Strange behaviors, as the displacement of the equilibrium states, are analyzed by means of the Gamma-convergence theory applied to the associated non-local Ginzburg-Landau energy functional. Finally, numerical experiments are shown in Section \ref{sec:numerical}.                  

\section{ A Fractional Semilinear Equation} \label{sec:AC_eq}
In order to study \eqref{eq:AC}, we temporally focus first on a ''restricted'' problem, 

\begin{equation}
	\left\lbrace
	\begin{array}{rl}      
		\ppa u + \eps^2 (-\Delta)^su  & = \fmonio(u)  \text{ in } \Omega \times (0,T), \\
		u(0)  & = v  \text{ in } \Omega, \\
		u & =  0  \text{ in }  \Omega^c\times[0,T],
	\end{array}
	\right.
	\label{eq:AC_trunc}
\end{equation}
where $g$ verifies the following conditions (H1) and (H2),  

\begin{equation} \text{(H1)} \qquad \fmonio \in C^2(\mathbb{R}),
	\label{H1}
\end{equation}

\begin{equation} \text{(H2)} \qquad  |\fmonio|,|\fmonio'|,|\fmonio''|<B \qquad \text{for some} \, B>0.
	\label{H2} 
\end{equation}

Clearly, the function $f$ of problem \eqref{eq:AC} does not comply with (H2). The goal is to apply later our results  to a problem  of the form \eqref{eq:AC_trunc} with a source term that, in addition to $(H1)$ and $(H2)$, agrees with $f$ in some interval  $[-1 - R , 1 + R]$, for an arbitrary $R>0$. In this case,  the condition $\|v\|_{L^{\infty}(\W)} \leq 1$ implies $\|u\|_{L^{\infty}(\W)} \leq 1$ which in turn allows to remove (H2) in this context. Therefore, for such initial condition, \eqref{eq:AC_trunc} and \eqref{eq:AC} are equivalents. This  $L^\infty$ bound is obtained indirectly  through the analysis of the semi-discrete in time scheme deferred to Section \ref{sec:bounds}.

\subsection{Weak formulation}
For any $s \in(0,1)$, we consider an open set $\W \subset \rn$. We  define the fractional Sobolev space $H^s(\W)$ as
\[
H^s(\W) = \left\{ v \in L^2(\W) \colon |v|_{H^s(\W)} := \left( \iint_{\W^2}  \frac{|v(x)-v(y)|^2}{|x-y|^{n+2s}} \, dx \, dy \right)^{\frac12} < \infty \right\}.
\]
This set, together with the norm $\|\cdot\|_{H^s(\W)} = \|\cdot\|_{L^2(\W)} + |\cdot|_{H^s(\W)} ,$ becomes a Hilbert space.

Another important space of interest for the problem under consideration is that of functions in $H^s(\rn)$ supported inside $\overline \W$, 
\[
\widetilde{H}^s (\W) = \left\{ v \in H^s(\rn) \colon \text{ supp } v \subset \bar{\W} \right\}.
\]	
The bilinear form 
\begin{equation}
	\label{eq:pinte}
	\langle u , v \rangle_{H^s(\rn)} := C(n,s) \iint_{(\rn \times \rn) \setminus (\W^c \times \W^c)} \frac{(u(x)-u(y)) (v(x)-v(y))}{|x-y|^{n+2s}} \, dy dx 
\end{equation}
constitutes an inner product on $\widetilde{H}^s(\W)$. The norm induced by the bilinear form, which is just a multiple of the $H^s(\rn)$-seminorm, is equivalent to the full $H^s(\rn)$-norm on this space, due to the fact that a Poincar\'e-type inequality holds in it. See, for example, \cite{AcostaBorthagaray} for details.

We call $u$ a \emph{weak solution} of \eqref{eq:AC_trunc}, if 
$u \in W^{1,1}((0,T),\ele)\cap C((0,T],\hsT)$ and  

\begin{equation}
	\label{eq:weakAC}
	\left\lbrace
	\begin{array}{rl}      
		(\ppa u,\phii) + \eps^2\langle u , \phii \rangle_{H^s(\rn)}  & = (\fmonio(u),\phii) \quad \forall \phii \in \hsT, \\
		u(0)  & = v  \text{ in } \Omega, 
	\end{array}
	\right.
\end{equation}
\emph{almost everywhere} in $(0,T)$. 

Let us note that, defining the operator $A: \hsT \subset \hmsT \to \hmsT$,

\begin{equation}
	\label{eq:weakOperator}
	(Au,\phii) = \langle u , \phii \rangle_{H^s(\rn)} \quad \forall \phii \in \hsT,
\end{equation}
the first identity of \eqref{eq:weakAC} can be rewritten as 

\begin{equation}
	\ppa u + \eps^2Au = \fmonio(u),
	\label{eq:weakAC2}
\end{equation}
a.e. in $(0,T)$.

\subsection{Solution representation}
For the  fractional eigenvalue problem, 
\begin{equation}
	\left\lbrace
	\begin{array}{rl}      
		(-\Delta)^s u  & = \lambda u \text{ in } \Omega \\
		u & =  0  \text{ in }    \mathbb{R}^n \setminus \Omega.
	\end{array}
	\right.
	\label{eq:eigenvalues}
\end{equation}
it is well-known that there exists a family of eigenpairs  $\{(\phi_k, \lambda_k) \}_{k=1}^\infty$, such that  
$$0 < \lambda_1 < \lambda_2 \le \ldots , \quad \lambda_k \to \infty \mbox{ as } k\to \infty ,
$$
with the eigenfunctions's set $\{ \phi_k \}_{k = 1}^\infty$ constituting an orthonormal basis of $L^2(\W)$. 
\begin{remark}
	Unlike eigenfunctions of the classical laplacian, solutions of \eqref{eq:eigenvalues} are in general non-smooth. Indeed, considering a smooth function  $d$ that behaves like $\delta(x)=\text{dist}(x,\partial \W)$ near to $\pp\W$, all eigenfunctions $\phi_k$ belong to the space $d^sC^{2s(-\eps)}(\W)$ (the $\eps$ is active only if $s=1/2$) and  $\frac{\phi_k}{d^s}$ does not vanish near $\partial\W$ \cite{grubb_autovalores, RosOtonSerra2}. Moreover,  
	the best Sobolev regularity guaranteed for solutions of \eqref{eq:eigenvalues} is $\phi_k \in H^{s + 1/2 - \eps}(\rn)$ for $\eps>0$ (see \cite{BdPM}).
\end{remark}

With $\{(\phi_k, \lambda_k) \}_{k=1}^\infty$, solutions of  \eqref{eq:eigenvalues}, and the Mittag-Leffler functions  $E_{\alpha,\mu} \colon \mathbb{C} \to \mathbb{C}$,
given by 
\begin{equation}
	\label{eq:Mittag-Leffler}
	E_{\alpha,\mu}(z) := \sum_{k=0}^\infty \frac{z^k}{\Gamma(\alpha k + \mu)},
\end{equation}
for each $\alpha > 0$ and $\mu \in \R$, we define the operators   

\begin{equation}
	\label{eq:operador_1}
	E^{\alpha}(t)v := \sum_{k} E_{\alpha,1}(-\lambda_k t^{\alpha}) \phi_k (v,\phi_k)_{\ele},
\end{equation}
and

\begin{equation}
	\label{eq:operador_2}
	F^{\alpha}(t)v := \sum_{k} t^{\alpha-1}E_{\alpha,\alpha}(-\lambda_k t^{\alpha}) \phi_k (v,\phi_k)_{\ele},
\end{equation}
for every $v \in \ele$. Following the theory for the linear case (see for instance \cite{ABB2}), the solution of \eqref{eq:AC_trunc} should, at least \emph{formally}, satisfy the integral equation 

\begin{equation}
	u(t) = E^{\alpha}(\eps^2 t)v + \int^{t}_0 F^{\alpha}(\eps^2(t-s) )\fmonio(u(s))\,ds.
	\label{eq:duhamelCaputo}
\end{equation}
We say that $u$ is a \emph{mild solution} of problem \eqref{eq:AC}, if $u$ is a solution of equation \eqref{eq:duhamelCaputo} and in this case, we use the notation $u(t) =: M(v,\fmonio)$.  

For technical purposes we define the interpolated norm

\begin{equation}
	\label{eq:norma_tita}
	\|w\|_{\theta,s} := \Big( \sum_{k} \lambda^{\theta}_k (w,\phi_k)^2_{\ele} \Big)^{\frac{1}{2}}. 
\end{equation} 
It can be easily verified that $\|w\|_{0,s} = \|w\|_{\ele}$, $\|w\|_{1,s} = |w|_{H^s(\rn)}$, and $\|w\|_{2,s} = \|(-\Delta)^{s} w\|_{\ele}$. Additionally, we denote $\dot{H}^{\theta}(\W) \subset \hmsT$, $\theta \geq -1$, the space induced by the norm \eqref{eq:norma_tita}.  

The following two lemmas provide helpful estimates for the operators \eqref{eq:operador_1} and \eqref{eq:operador_2}.

\begin{lemma}
	\label{estimacionSemigrupo}
	Consider $t>0$, then we have
	
	\begin{align}
		\|E^{\alpha}(t)v\|_{p,s}  \leq Ct^{-\alpha(p-q)/2}\|v\|_{q,s}, & \quad \text{if } 0 \leq p - q \leq 2, \label{eq:semigEst1} \\
		\|F^{\alpha}(t)v\|_{p,s}  \leq Ct^{-1 + \alpha(1 + (q-p)/2)}\|v\|_{q,s}, & \quad \text{if } 0 \leq p - q \leq 4, \label{eq:semigEst2}
	\end{align}

\end{lemma}

\begin{proof} The proof is analogous of that of  \cite[Lemma 2.2]{BLPZ} (see also \cite[Lemma 2.0.2]{tesis_yo}).  \end{proof}

\begin{lemma}\label{lem:derivada_operador}
	If $v\in \dot{H}^ q$, $q\in[0,2]$, then for $m\geq1$
	\begin{equation*}
		\| \partial_t^m E^{\alpha}(t) v \|_{L^2(\W)} \le C t^{q\alpha/2-m}  \| v \|_{q,s}.
	\end{equation*}
\end{lemma}
\begin{proof} The proof can be carried out as in \cite[Theorem A.2]{Jin} (see also \cite[Lemma 2.0.3]{tesis_yo}). \end{proof}


\subsection{ Existence and Uniqueness}

For the sake of simplicity we are going to consider $\eps^2=1$ along this section.

To obtain an existence and uniqueness result for the integral equation \eqref{eq:duhamelCaputo}, we use standard fixed point arguments adapted to the present context. Our approach follows \cite{tesisneto} and \cite{larsson}. For $q \in (0,1]$ and $\tau>0$, we introduce the space,
$$\V_{\tau} := \{ w \in C([0,\tau],\ele) \cap C^1((0,\tau],\ele) \text{ such that } \|w\|_{\V_{\tau}}<\infty\}$$
where $\|w\|_{\V_{\tau}}$ is defined as 
$$\|w\|_{\V_{\tau}} := \sup_{t \in [0,\tau]} \|w(t)\|_{\ele} + \sup_{t \in [0,\tau]} t^{(1-q)\alpha/2}|w(t)|_{H^s(\rn)} +  \sup_{t \in [0,\tau]}t^{1-q\alpha/2}\|\partial_t w(t)\|_{\ele}.$$ 
The inclusion $\V_{\tau} \subset \mathbb{V}_{\tau}^{q'}$, for $q'<q$, follows  immediately from the definition, while the fact that $\V_{\tau}$ is a Banach space can be proved by means of standard arguments.  In the sequel, the parameter $q$ plays a role in connection with the regularity of the initial datum $v$. In particular, if $\|v\|_{q,s} < \infty$ for some positive $q$, we show  below that  $ \partial_t u \in L^{1}((0,T),\ele)$. This condition on $\partial_t u$ is important for the analytical treatment of the numerical error and a fundamental assumption for a right  definition of the Caputo operator \eqref{eq:caputo}.

\begin{remark}
	It is important to observe that we necessarily need $v \in \dot{H}^{q}(\W)$ for some positive $q$. Otherwise, the fractional derivative in time could not be well defined for solutions of problem \eqref{eq:AC}. In fact, even considering a simpler problem, taking $f \equiv 0$ in \eqref{eq:AC}, it is possible to construct a function $v \in L^2(\W)$ in such a way that $v \not \in \dot{H}^{q}(\W)$ for all $q>0$ and, for that initial datum, the solution $u \not \in W^{1,1}((0,T),L^2(\W))$. We refer to \cite[Remark 2.1.5]{tesis_yo} for details.  
\end{remark}

The following is  a local existence result.

\begin{theorem}
	\label{EUlocal}
	Suppose that $\|v\|_{q,s} \leq R_0$ for some $R_0 > 0$ and $q \in (0,1]$. Then, there exist $\tau>0$ small enough, such that equation \eqref{eq:duhamelCaputo} has a unique solution $u \in \V_{\tau}$. 
\end{theorem}

\begin{proof}
	
	First, we define the operator $\mathcal{S}(u)$
	
	\begin{equation}
		\mathcal{S}(u)(t) := E^{\alpha}(t)v + \int^{t}_0 F^{\alpha}(t-s)\fmonio(u(s))\,ds, 
		\label{eq:puntoFijo}
	\end{equation}
	and $B_{R} = \{ w \in \V_{\tau} \, \text{ such that } \|w\|_{\V_{\tau}} \leq R, \text{ and } w(0) \equiv v \}$. It can be easily verified that $B_{R} \subset \V_{\tau}$ is a closed set. Our goal is to show that there are parameters $\tau>0$ and $R>0$, in such a way that we can apply Banach's fixed point theorem. That is, we look for $\tau$ and $R$, such that $\mathcal{S}$ maps $B_R$ into itself, and results in a contraction over $B_R$.
	
	Indeed, observing first that $\mathcal{S}(u)(0) \equiv v$ for all $u \in \V_{\tau}$, then the condition $u(0) \equiv v$ is satisfied for every output of $\mathcal{S}$. Furthermore, by means of Lemma \ref{lem:leinbiz}, it can be seen that $\mathcal{S}(u)(t) \in C([0,\tau],\ele) \cap C^1((0,\tau],\ele)$. Suppose now $u \in B_R$, from \eqref{eq:puntoFijo}, Lemma \ref{estimacionSemigrupo} and the definition of $\fmonio$, we have  
	
	\begin{equation*}
		t^{(1-q)\alpha/2}| \mathcal{S}(u)(t) |_{H^s(\rn)} \leq 
	\end{equation*}
	$$
	t^{(1-q)\alpha/2}|E^{\alpha}(t)v|_{H^s(\rn)} + t^{(1-q)\alpha/2}\int^{t}_0 |F^{\alpha}(t-s)\fmonio(u(s))|_{H^s(\rn)}\,ds \leq 
	$$
	$$C\|v\|_{q,s} + Ct^{(1-q)\alpha/2}\int^{t}_0 (t-s)^{\alpha/2 - 1}\|\fmonio(u(s))\|_{\ele}\,ds, $$
	using the boundedness of $g$ and computing the resulting    integral, together with the fact $t<\tau$, we get 
	
	\begin{equation}
		t^{(1-q)\alpha/2}| \mathcal{S}(u)(t) |_{H^s(\rn)} \leq CR_0 + C\tau^{\alpha}.
		\label{eq:puntoFijo2}
	\end{equation}
	With the same idea we can obtain
	
	\begin{equation}
		\label{eq:puntoFijo2bis}
		\| \mathcal{S}(u)(t) \|_{\ele} \leq CR_0 + C\tau^{\alpha}
	\end{equation}
	On the other hand, by means of Lemma \ref{lem:leinbiz}, we have
	
	\begin{equation}
		\label{eq:derivada_semi}
		\partial_t \left( \int^{t}_0 F^{\alpha}(t-s)\fmonio(u(s)) \,ds \right) = \partial_t \left( \int^{t}_0 F^{\alpha}(s)\fmonio(u(t-s)) \,ds \right)
	\end{equation}
	$$= F^{\alpha}(t)\fmonio(v) + \int^{t}_0 F^{\alpha}(s)\fmonio'(u(t-s))\partial_t u(t-s)  \,ds, $$ 
	and we can write
	
	\begin{equation*}
		t^{1-q\alpha/2}\| \partial_t \mathcal{S}(u)(t) \|_{\ele} \leq t^{1-q\alpha/2}\|\partial_t E^{\alpha}(t)v\|_{\ele}+ t^{1-q\alpha/2}\| F^{\alpha}(t)\fmonio(v)\|_{\ele} 
	\end{equation*}
	$$+ t^{1-q\alpha/2}\int^{t}_0 \|F^{\alpha}(s)\fmonio'(u(t-s))\partial_t u(t-s)\|_{\ele}\,ds $$
	$$\leq CR_0 + CRt^{1-q\alpha/2}\int^{t}_0 s^{\alpha - 1}(t-s)^{q\alpha/2-1}\,ds  $$
	where we have applied Lemma \ref{lem:derivada_operador}, Lemma \ref{estimacionSemigrupo}, the fact that $t^{1-q\alpha/2}\|\partial_t u(t)\|_{\ele} \leq \|u\|_{\V_{\tau}} \leq R,$
	and $t^{1-q\alpha/2}\|F^{\alpha}(t)\fmonio(v)\|_{\ele} \leq \tau^{(1-q/2)\alpha}\|v\|_{\ele}$.
	The integral in the second term can be estimated in terms of the beta function $B(\alpha,q\alpha/2)$. Indeed, making the change of variables $s/t = r$, we obtain 
	
	\begin{equation}
		t^{1-q\alpha/2}\| \partial_t \mathcal{S}(u)(t) \|_{\ele} \leq CR_0 + CRt^{\alpha}\int^{1}_0 r^{\alpha - 1}(1-r)^{q\alpha/2-1}\,dr 
		\label{eq:puntoFijo2.1}
	\end{equation}
	$$ \leq CR_0 + CRB(\alpha,q\alpha/2)\tau^{\alpha}. $$

	Combining \eqref{eq:puntoFijo2}, \eqref{eq:puntoFijo2bis} and \eqref{eq:puntoFijo2.1}, we have 
	
	$$\|\mathcal{S}(u)\|_{\V_{\tau}} \leq CR_0 + CR\tau^{\alpha}, $$
	where $C=C(\alpha)$. 
	Then, fixing $R = 2CR_0$, we can choose $\tau>0$ small enough to satisfy the inequality $\| \mathcal{S}(u) \|_{\V_{\tau}} < R$ . Hence, for this $\tau$, $\mathcal{S}$ maps $B_R$ into itself.
	
	Now we want to see that $\mathcal{S}$ is a contraction over $B_R$. Indeed, let $u$ and $w \in B_{R}$, using $|\fmonio'| \leq B$ and Lemma \ref{estimacionSemigrupo}, we have 
	
	\begin{equation}
		\label{eq:contraccion}
		t^{(1-q)\alpha/2}|\mathcal{S}(u)(t) - \mathcal{S}(w)(t)  |_{H^s(\rn)} \leq 
	\end{equation}
	
	$$
	Ct^{(1-q)\alpha/2}\int^{t}_0 (t-s)^{\alpha/2 - 1}\|\fmonio(u(s)) - \fmonio(w(s))\|_{\ele}\,ds \leq 
	$$
	
	$$\leq CBt^{(1-q)\alpha/2}\int^{t}_0 (t-s)^{\alpha/2 - 1}\|u(s) - w(s)\|_{\ele}\,ds  $$
	
	$$\leq \|u - w\|_{\V_{\tau}}CBt^{(1-q)\alpha/2}\int^{t}_0 (t-s)^{\alpha/2 - 1}\,ds $$
	
	$$ \leq Ct^{(2-q)\alpha} \|u - w\|_{\V_{\tau}} \leq C\tau^{\alpha} \|u - w\|_{\V_{\tau}}.$$

	With similar arguments it can be seen that
	
	\begin{equation}
		\label{eq:contraccionBis}
		\|\mathcal{S}(u)(t) - \mathcal{S}(w)(t)  \|_{\ele} \leq C \tau^{\alpha}\|u - w\|_{\V_{\tau}}.
	\end{equation}
	
	Recalling the equality \eqref{eq:derivada_semi}, we have that 
	
	$$
	t^{1-q\alpha/2}\|\partial_t \big( \mathcal{S}(u)(t) - \mathcal{S}(w)(t)  \big) \|_{\ele} \leq t^{1-q\alpha/2}\|F^{\alpha}(t)\big(u(0) - w(0)\big)\|_{\ele} 
	$$
	
	$$+t^{1-q\alpha/2}C\int^{t}_0 s^{\alpha - 1}\|\fmonio'(u(t-s))\partial_t u(t-s) - \fmonio'(w(t-s))\partial_t w(t-s)\|_{\ele}\,ds. $$
	Using the identity
	
	$$\fmonio'(u)\partial_t u - \fmonio'(w)\partial_t w = \fmonio'(u)(\partial_t u - \partial_t w)  - (\fmonio'(u) - \fmonio'(w)) \partial_t w,$$
	and the fact that $u(0)\equiv w(0) \equiv v$, $t^{1-q\alpha/2}\|\partial_t w(t)\|_{\ele} \leq R$, $|\fmonio'|$, $|\fmonio''| \leq B$, we can write 
	
	\begin{equation}
		\label{eq:contraccion2}
		t^{1-q\alpha/2}\|\partial_t \big( \mathcal{S}(u)(t) - \mathcal{S}(w)(t)  \big) \|_{\ele} \leq
	\end{equation}
	$$
	t^{1-q\alpha/2}CB\int^{t}_0 s^{\alpha - 1} \|\partial_t\big( u(t-s)- w(t-s) \big) \|_{\ele} \,ds
	$$
	$$+ CBRt^{1-q\alpha/2}\int^{t}_0 s^{\alpha - 1} (t-s)^{\alpha/2-1}\|u(t-s)- w(t-s) \|_{\ele} \,ds$$
	
	$$\leq C(1+R) \|u-w\|_{\V_{\tau}} t^{1-q\alpha/2}\int^{t}_0 s^{\alpha - 1} (t-s)^{q\alpha/2-1} \,ds$$
	
	$$ \leq  C B(\alpha,q\alpha/2)\tau^{\alpha}\|u-w\|_{\V_{\tau}},$$
	where the integrals in the last inequality have been estimated in terms of the beta function, as in \eqref{eq:puntoFijo2.1}, and $C=C(R)$.

	Finally, combining \eqref{eq:contraccion}, \eqref{eq:contraccionBis} and \eqref{eq:contraccion2}, we can conclude that 
	
	$$\|\mathcal{S}(u)(t) - \mathcal{S}(w)(t)  \|_{\V_{\tau}} \leq  C\tau^{\alpha}\|u-w\|_{\V_{\tau}}, $$ 
	with $C=C(\alpha,R)$, and it is clear that we can choose $\tau$ small enough, such that $\mathcal{S}$ results in a contraction over $B_{R}$. Hence, for that $\tau$,  a unique solution for problem \eqref{eq:AC} in the interval $[0,\tau]$.   \end{proof}

Now we need to derive an a priori estimate for the time derivative of the solution. To this end, we first recall the following Gronwall type inequality .

\begin{lemma}
	\label{gronwall}
	Let the function $\phii(t)\geq 0 $ be continuous for $0 < t \leq T$. Then, if
	
	$$\phii(t) \leq At^{-1+\alpha} + B\int^t_0 (t-s)^{-1+\beta}\phii(s) \, ds \quad 0<t\leq T$$  
	for some constants $A$,$B \geq 0$ and $\alpha, \beta > 0$, there exists a constant $C = C(B,T,\alpha,\beta)$ such that 
	
	\begin{equation}
		\label{eq:gronwall}
		\phii(t) \leq CAt^{-1+\alpha}
	\end{equation}
\end{lemma}
\begin{proof}
	See, for instance, \cite{elliott} Lemma 6.3.
\end{proof}
Now, we are ready to state the following result.
\begin{lemma}
	\label{lem:estimacion_derivada}
	Let $u(t) = M(v,\fmonio)(t)$ with $v \in \dot{H}^{q}(\W)$ and $t \in [0,T]$, there exists a constant $C = C(\alpha,T)$ such that 
	
	\begin{equation}
		\label{eq:estimacion_derivada}
		\|\partial_t u(t)\|_{\ele} \leq C t^{\alpha/2 - 1}
	\end{equation}
	
\end{lemma}   

\begin{proof}

	For $h>0$ we can write
	\begin{equation}
		\label{eq:est_derivada2}
		u(t+h) - u(t) = \big(E^{\alpha}(t+h)-E^{\alpha}(t)\big)v + \int^{t+h}_0F^{\alpha}(t+h-s)\fmonio(u(s)) \, ds  
	\end{equation}
	$$- \int^{t}_0 F^{\alpha}(t-s)\fmonio(u(s)) \, ds$$
	
	$$= \big(E^{\alpha}(t+h)-E^{\alpha}(t)\big)v + \int^{t+h}_0F^{\alpha}(s)\fmonio(u(t+h-s)) \, ds - \int^{t}_0 F^{\alpha}(s)\fmonio(u(t-s)) \, ds$$
	
	
	$$
	=  \big(E^{\alpha}(t+h)-E^{\alpha}(t)\big)v + \int^{t+h}_t F^{\alpha}(s)\fmonio(u(t+h-s)) \, ds  
	$$
	
	$$
	+ \int^{t}_0 F^{\alpha}(s) \big( \fmonio(u(t+h-s)) - \fmonio(u(t-s)) \big) \, ds
	$$
	
	$$
	=  \big(E^{\alpha}(t+h)-E^{\alpha}(t)\big)v + \int^{t+h}_t F^{\alpha}(s)\fmonio(u(t+h-s)) \, ds $$
	
	$$
	+ \int^{t}_0 F^{\alpha}(t-s) \big( \fmonio(u(s+h)) - \fmonio(u(s)) \big) \, ds.
	$$
	
	Now, considering $h$ small enough, and taking norms at both sides of the equality; using Lemma \ref{lem:derivada_operador} in the first term on the left side; estimation \eqref{eq:semigEst2}, and estimation $|\fmonio|<B$ in the second term and the same idea in the last one, we obtain
	
	$$\|u(t+h) - u(t)\|_{\ele} \leq $$
	$$
	C\Big( h t^{q\alpha/2 - 1} + \int^{t+h}_{t} s^{\alpha-1} \, ds + \int^{t}_{0} (t-s)^{\alpha-1}\|u(s+h) - u(s)\|_{\ele} \, ds \Big)
	$$
	\begin{equation}
		\label{eq:est_derivada3}
		\leq C(T)\Big( h t^{q\alpha/2 - 1} + \int^{t}_{0} (t-s)^{\alpha-1}\|u(s+h) - u(s)\|_{\ele} \, ds \Big).
	\end{equation}
	Finally, applying Lemma \ref{gronwall} we derive \eqref{eq:estimacion_derivada}.      
	
\end{proof}

Combining the former results, we are now able to prove the global existence of the solution.

\begin{theorem}
	\label{teo_ex_global}
	Under the hypotheses of Theorem \ref{EUlocal}, let  $u$ be the solution of \eqref{eq:puntoFijo} defined in $[0,\tau]$ and consider fixed numbers $T$ and $\tau_0$, such that $T>\tau>t_0$.  Then, there exists a constant $C = C(T, \tau_0) >0$ such that if $0<\delta \leq C$, $u$ can be extended to $[0,\tau + \delta]$ as a solution of \eqref{eq:puntoFijo}.   
\end{theorem}

\begin{proof}
	We are going to consider the space $\V_{\tau+\delta}$, for some $0 < \delta < 1$, and  $B_{R} \subset \V_{\tau+\delta}$, defined as $B_R := \{ w \in \V_{\tau+\delta} \text{ such that } w(t) \equiv u(t) \, \forall t \in [0,\tau], \text{ and } \|w\|_{\V_{\tau+\delta}} \leq R  \}$, where $u$ is the solution of \eqref{eq:puntoFijo} over $[0,\tau]$. Observe that, with this definition, $B_R$ is a closed subset of $\V_{\tau+\delta}$. Our goal is, as in the proof of Theorem \ref{EUlocal}, to apply Banach's fixed point Theorem, showing that there exist $\delta > 0$ and $R$, such that $\mathcal{S}$ is a contraction over $B_R$, and maps $B_R$ into itself. 
	
	Suppose $\tilde{u} \in B_R$, proceeding similarly as in \eqref{eq:puntoFijo2}, using the boundedness of $\fmonio$, we can obtain
	
	\begin{equation}
		\label{eq:global1}
		t^{(1-q)\alpha/2}|\mathcal{S}(\tilde{u})|_{H^s(\rn)} \leq CR_0 + C(\tau+\delta)^{\alpha} \leq C(R_0 + \tau^{\alpha}+\delta^{\alpha}) 
	\end{equation}
	$$ \leq C(R_0,T) + \delta^{\alpha}.$$
	With the same idea we obtain 
	\begin{equation}
		\label{eq:global1bis}
		\|\mathcal{S}(\tilde{u})\|_{\ele} \leq C(R_0,T) + \delta^{\alpha}.
	\end{equation}

	Also, applying the same arguments used to arrive to \eqref{eq:puntoFijo2.1}, along with the fact that $\tilde{u}(s) = u(s)$ for all $s \in [0,\tau]$ together with the fact that $t>\tau$, we get 
	
	\begin{equation}
		\label{eq:global2}
		t^{1-q\alpha/2}\|\partial_t \mathcal{S}(\tilde{u})\|_{\ele} = 
	\end{equation} 
	$$
	t^{1-q\alpha/2}\| \partial_t E^{\alpha}(t)v+ F^{\alpha}(t)\fmonio(v) + \int^{t}_{0} F^{\alpha}(s) \fmonio(\tilde{u}(t-s))\partial_t \tilde{u}(t-s)\,ds \|_{\ele} \leq 
	$$
	$$
	C(T)R_0 + t^{1-q\alpha/2}\|\int^{t}_{0} F^{\alpha}(t-s) \fmonio(\tilde{u}(s))\partial_t \tilde{u}(s)\,ds\|_{\ele} \leq
	$$
	$$
	C(T)R_0 + t^{1-q\alpha/2}\|\int^{\tau}_{0} F^{\alpha}(t-s) \fmonio(u(s))\partial_t u(s)\,ds + \int^{t}_{\tau} F^{\alpha}(t-s) \fmonio(\tilde{u}(s))\partial_t \tilde{u}(s)\,ds\|_{\ele}
	$$
	$$
	\leq C(T)R_0 
	+ CBt^{1-q\alpha/2}\int^{\tau}_{0} (t-s)^{\alpha-1} s^{q\alpha/2-1}\,ds 
	+ CBRt^{1-q\alpha/2}\int^{t}_{\tau} (t-s)^{\alpha-1} s^{q\alpha/2-1}\,ds
	$$
	$$
	= C(R_0,T)  + (i) + (ii),
	$$
	where in the last inequality we have used \eqref{eq:estimacion_derivada}. Now, making the change of variables $s/t = r$, we have

	$$
	(i) \leq Ct^{\alpha}\int^{\tau/t}_{0} (1-r)^{\alpha-1} r^{q\alpha/2-1}\,dr \leq 
	C(\tau + \delta)^{\alpha} \int^{1}_{0} (1-r)^{\alpha-1} r^{q\alpha/2-1}\,dr \leq C\tau^{\alpha} + C\delta^{\alpha},
	$$
	$$\leq C(T) + C\delta^{\alpha} $$
	and
	$$
	(ii) \leq CRt^{\alpha}\int^{1}_{\tau/t} (1-r)^{\alpha-1} r^{q\alpha/2-1}\,dr \leq 
	CRt^{\alpha}  (\tau/t)^{q\alpha/2-1}\int^{1}_{\tau/t} (1-r)^{\alpha-1}\,dr
	$$
	$$
	\leq CRt^{\alpha}  (\tau/t)^{q\alpha/2-1}(1-\tau/t)^{\alpha} \leq CR(t-\tau)^{\alpha} \leq CR\delta^{\alpha},
	$$
	where we have estimated $(\tau/t)^{q\alpha/2-1}<C(\tau_0)$ using the fact that $t \geq \tau>\tau_0>0$.  
	
	Applying this estimation to \eqref{eq:global2},  we obtain
	
	\begin{equation}
		\label{eq:global3}
		t^{1-q\alpha/2}\|\partial_t \mathcal{S}(\tilde{u})\|_{\ele} \leq C(R_0,T) +  CR\delta^{\alpha}, 
	\end{equation}  
	and combining \eqref{eq:global3} with \eqref{eq:global1} and \eqref{eq:global1bis}, we obtain
	
	$$ \| \mathcal{S}(\tilde{u})\|_{\V_{\tau+\delta}} \leq C(R_0,T) +  CR\delta^{\alpha}.$$
	If we choose $R = 2C(R_0,T)$, taking $\delta^{\alpha}\leq 1/2C$ we have $\| \mathcal{S}(\tilde{u})\|_{\V_{\tau+\delta}} \leq R$. 
	
	Finally, we only need to show that $\mathcal{S}$ is a contraction on $B_R$. Consider $\tilde{u}$ and $w \in \V_{\tau+\delta}$, proceeding as in \eqref{eq:contraccion}, and taking advantage of the fact that $\tilde{u}(s)=w(s)=u(s)$ for all $s \in [0,\tau]$, we can estimate

	\begin{equation}
		\label{eq:contraccion_gblogal}
		t^{(1-q)\alpha/2}|\mathcal{S}(\tilde{u})(t) - \mathcal{S}(w)(t)  |_{H^s(\rn)} \leq   
	\end{equation}
	
	$$
	Ct^{(1-q)\alpha/2}\int^{t}_{\tau} (t-s)^{\alpha/2 - 1}\|\fmonio(\tilde{u}(s)) - \fmonio(w(s))\|_{\ele}\,ds
	$$
	
	$$\leq CBt^{(1-q)\alpha/2}\int^{t}_{\tau} (t-s)^{\alpha/2 - 1}\|\tilde{u}(s) - w(s)\|_{\ele}\,ds  $$
	
	$$\leq CB\|\tilde{u} - w\|_{\V_{\tau + \delta}} t^{(2-q)\alpha/2}\int^{1}_{\tau/t} (1-r)^{\alpha/2 - 1}\,dr  $$
	
	$$\leq C\|\tilde{u} - w\|_{\V_{\tau+\delta}}t^{-q\alpha/2}t^{\alpha/2}(1-\tau/t)^{\alpha/2} \leq C\delta^{\alpha/2}\|\tilde{u} - w\|_{\V_{\tau + \delta}}, $$
	where in the last step we use the bound  $t^{-q\alpha/2} \leq C(\tau_0)$, with $\tau > \tau_0 > 0$.

	Also, arguing as in \eqref{eq:contraccion2}, we have
	\begin{equation}
		\label{eq:glogal4}
		t^{1-q\alpha/2}\|\partial_t \big( \mathcal{S}(\tilde{u})(t) - \mathcal{S}(w)(t)  \big) \|_{\ele} \leq 
	\end{equation}
	$$
	t^{1-q\alpha/2}CB\int^{t}_{\tau} (t-s)^{\alpha - 1} \|\partial_t\big( \tilde{u}(s)- w(s) \big) \|_{\ele} \,ds
	$$
	$$+ CBRt^{1-q\alpha/2}\int^{t}_{\tau} (t-s)^{\alpha - 1} s^{q\alpha/2-1}\|\tilde{u}(s)- w(s) \|_{\ele} \,ds$$
	
	$$\leq C(R+1) \|\tilde{u}-w\|_{\V_{\tau + \delta}} t^{1-q\alpha/2}\int^{t}_{\tau} (t-s)^{\alpha - 1} s^{q\alpha/2-1} \,ds$$

	$$\leq C(R+1) \|\tilde{u}-w\|_{\V_{\tau + \delta}} t^{\alpha}\int^{1}_{\tau/t} (1-r)^{\alpha - 1} r^{q\alpha/2-1} \,dr$$
	with the same arguments used to bound $(ii)$ we arrive to

	$$ \leq  C(R+1) \delta^{\alpha}\|\tilde{u}-w\|_{\V_{\tau + \delta}}.$$ 
	
	Then, we can assert that  $\|\mathcal{S}(\tilde{u}) - \mathcal{S}(w) \|_{\V_{\tau+\delta}} \leq C\delta^{\alpha}(R+1) \|\tilde{u}-w\|_{\V_{\tau + \delta}}, $
	and we can choose $\delta$ such that $\mathcal{S}$ results in a contraction. Since $R$ depends on $T$ and $R_0$, the statement of the theorem follows.\end{proof}

Notice, in previous Theorem, that $\delta$ does not depend on $\tau$. As a consequence, we have proved that equation \ref{eq:duhamelCaputo} has a unique solution in $\V_T$. Moreover, in view of the regularity of functions belonging to the space $\V_{T}$, we can assert that a mild solution is also a weak solution.


\section{Numerical Scheme} \label{sec:ns}

\subsection{Semi-discrete Scheme}
Let $\mathcal{T}_{h}$ be a shape regular and quasi-uniform admissible triangulation of $\Omega$. With $X_{h} \subset \widetilde{H}^s(\Omega)$ we denote the continuous piecewise linear finite element space associated with $\mathcal{T}_h$, that is,
$$X_{h} := \{u_h \in \widetilde{H}^s(\Omega)\cap C^0(\bar \Omega) \colon \restr{u_h}{T} \in \mathcal{P}^1 \ \forall T \in \mathcal{T}_{h}  \} .$$
Then, semi-discrete problem formulation reads:  find $u_h \colon [0, T] \to X_h$ such that 
\begin{equation} 
	\label{semiDiscreto}
	\left\lbrace
	\begin{array}{rl}
		( \ppa u_h  , w )  +  \langle u_h, w \rangle_{H^s(\rn)} & =  \left( f(u_h), w \right), \quad \forall w \in X_h, \\
		u_h(0) & = v_h .
	\end{array}
	\right.
\end{equation}
Here, $v_h = P_h v$, and $P_h$ denotes the $L^2(\W)$ projection on $X_h$. 
Also, defining the discrete fractional Laplacian $A_{h}: X_h \rightarrow X_{h}$ as the unique operator that satisfies $ ( A_{h} w , v ) = \langle w , v \rangle_{H^s(\rn)}, \text{  for all } w,v \in X_h, $
we may rewrite \eqref{semiDiscreto} as
\begin{equation} 
	\left\lbrace
	\begin{array}{rl}
		\ppa u_h   +  \eps^2A_h u_h & =  P_hf(u_h), \\
		u_h(0) & = v_h.
	\end{array}
	\label{eq:semiDiscreto2}
	\right.
\end{equation}  

We also define the discrete versions of $E^{\alpha}$ and $F^{\alpha}$. In order to do this, consider an orthonormal basis of $X_h$,  $\{\phi_{h,1},...,\phi_{h,N}\} \subset X_h$ and define

\begin{equation}
	\label{eq:operador_1_discreto}
	E_h^{\alpha}(t)v := \sum^{N}_{k=1} E_{\alpha,1}(-\lambda_{h,k} t^{\alpha}) \phi_{h,k} (v,\phi_{h,k})_{\ele},
\end{equation}
and
\begin{equation}
	\label{eq:operador_2_discreto}
	F_h^{\alpha}(t)v := \sum^N_{k=1} t^{\alpha-1}E_{\alpha,\alpha}(-\lambda_{h,k} t^{\alpha}) \phi_{h,k} (v,\phi_{h,k})_{\ele}.
\end{equation}

\subsection{Discretizing the Caputo derivative}

In order to set a fully discrete scheme, we need to discretize the Caputo operator. This can be done by means of the well known relation between the Caputo derivative $\ppa$, and the Riemann-Liouville operator $\rl$, that reads

\begin{equation}
	\label{eq:caputo2rl}
	\ppa u(t) = \rl \Big( u(t) - u(0) \Big),
\end{equation}
for $0<\alpha<1$ (see, for instance, \cite{Diethelm} Theorem 3.1) that holds for a smooth enough function $u$. Applying \eqref{eq:caputo2rl} to problem \eqref{eq:semiDiscreto2}, we can reformulate it in terms of the Riemann-Liouville operator as follow,

\begin{equation} 
	\left\lbrace
	\begin{array}{rl}
		\rl u_h   +  \eps^2A_h u_h & = \rl v_h + P_hf(u_h), \\
		u_h(0) & = v_h.
	\end{array}
	\label{eq:semiDiscreto2rl}
	\right.
\end{equation} 
The advantage here is that the R-L derivative can be approximated by means of a convolution quadrature rule. That is, dividing the interval $[0,T]$ uniformly with time step $\tau = T/N$, and letting $t_n = n\tau$, a discrete estimation $\pp^{\alpha}_{\tau}$ of $\rl$ can be defined as

\begin{equation}
	\label{conv_dis}
	\pp^{\alpha}_{\tau}u(t_n) = \sum^{n}_{j = 0} \w_j u(t_n - j\tau). 
\end{equation}
Here, the weights $\{\w_j\}_{j \in \N}$ are obtained as the coefficients of the power series expansion $\Big(\frac{1 - \xi}{\tau}\Big)^{\alpha} = \sum^{\infty}_{j=0} \w_j \xi^j. $
Fast Fourier Transform can be used for an efficient computation of $\{\w_j\}_{j\in\N_0}$, (see \cite{Podlubny} Section 7.5). Alternatively, a useful recursive expression is also given in \cite{Podlubny},

\begin{equation}
	\label{eq:definicion_pesos}
	\w_0 = \tau^{-\alpha}, \quad \w_j = \Big(1 - \frac{\alpha + 1}{j}\Big)\w_{j-1}, \quad \forall j >0.
\end{equation}

It is not our intention to give an exhaustive description of this method and we refer the reader to \cite{ABB2,Jin} (see also   \cite{Lub2,Lub1} for further details). An advantage of convolution quadratures  is that  error estimates can be delivered without the assumption of excessively restrictive regularity properties on the solution. This is a fact of paramount importance as one can learn from  \cite{stynes}.

The following result  (Theorem 5.2, \cite{Lub2}) will play a central role in the error estimation below.

\begin{lemma}
	\label{lem:est_lubich}
	Let $K$ be a complex valued or operator valued function which is analytic in a sector $\Sigma_{\theta}:=\{z \in \mathbb{C}\, : \, |\arg z| \leq \theta \}$, with $\theta \in (\pi/2 , \pi )$, and bounded by $\|K(z)\| \leq M|z|^{-\mu} \, \, \forall z \in \Sigma_{\theta},$
	for some $\mu, M \in \mathbb{R}$. Then for $g(t) = Ct^{\beta-1}$, the operator $\pp_{\tau}$ satisfies 
	
	\begin{equation*}
		\| (K(\dt) - K(\dtd))g(t)\| \leq
		\left\lbrace
		\begin{array}{ll}
			ct^{\mu -1}\tau^{\beta}, & 0<\beta \leq 1, \\
			ct^{\mu + \beta -2}\tau, & \beta \geq 1 .
		\end{array}
		\right.
		\label{estimacion_lubich}
	\end{equation*}  
	
\end{lemma}  

Finally, another useful property of the operator $\pp_{\tau}$ is the associativity. That is, let $K_1, K_2$ be operators as in Lemma \ref{lem:est_lubich}, and $k$ an analytic function, we have 

\begin{equation}
	\label{eq:asociatividad}
	K_1(\pp_{\tau})K_2(\pp_{\tau}) = (K_1 K_2)(\pp_{\tau}) \quad \text{ and } K_1(\pp_{\tau})(k*g) = (K_1(\pp_{\tau})k)*g. 
\end{equation}

\subsection{Fully discrete scheme}

Replacing the Riemann-Liouville derivative by its discrete version given by \eqref{conv_dis}, we can formulate the fully discrete problem as: find $U^n_{h} \in X_h$, with $n = \{1,\ldots,N \}$, such that  
\begin{equation} 
	\left\lbrace
	\begin{array}{rl}
		\dtd^{\alpha} U_h^n   +  A_h U_h^n & =  \dtd^{\alpha}v_h +  P_h \fmonio(U^n_h) \\
		U_h^0 & = v_h, 
	\end{array}
	\label{fully_AC_rl}
	\right.
\end{equation}

For the sake of the reader's convenience, we include a vectorial form of the fully discrete scheme. Let $\{ \varphi_i \}_{i=1,\ldots,\mathcal{N}}$ be the Lagrange nodal basis that generates $X_h$. Let $U^n \in \mathbb{R}^{\mathcal{N}}$, $n=0,\ldots,N$ be such that $U^n_h = \sum_{i=1}^{\mathcal{N}} U^n_i\varphi_i$, where $U_h^n$ denotes the solution of the fully discrete problem. 
Then, we may formulate \eqref{fully_AC_rl} in the following vectorial non-linear equation:
\begin{equation*} 
	M^{-1}\cdot(\w_0 M  +  K) \cdot U^n  =  \left(\sum_{j=0}^{n}\w_j \right)U^{0} - \sum^n_{j=1}\w_jU^{n-j} +  \fmonio(U^n).
	\label{full_vect}
\end{equation*}
Where $M$ and $K$ are the mass and stiffness matrices respectively. That is, $M_{i,j} = (\varphi_i,\varphi_j)$ and $K_{i,j} = \langle \varphi_i,\varphi_j \rangle_{H^s(\rn)}$.

The computation and assembly of the stiffness matrix in dimension greater than one is not a trivial task. Nevertheless, this problem for two-dimensional domains is treated in \cite{ABB}, where the authors provide a short MATLAB implementation to this end. Also we can mention \cite{glusa,KM} where some clever ways to reduce the complexity of the assembling process are analyzed. 

Since \eqref{fully_AC_rl} is not a linear equation, it is not clear a priori that there exist a solution. In that way, next result gives us existence and uniqueness for problem \eqref{fully_AC_rl}. 

\begin{theorem}
	\label{teo:solucion_fully_dis}
	There exist $\tau$ small enough, such that problem \eqref{fully_AC_rl} has  a unique solution $U^n_h \in X_h$ for all $n \in \{1,...,n\}$. 
\end{theorem}     

\begin{proof}
	Recalling that $\w_0 = \tau^{-\alpha}$, dividing equation \eqref{fully_AC_rl} by $\w_0$ on both sides, we obtain 
	
	\begin{equation} 
		(I  +  \tau^{\alpha}A_h) U_h^n  =  \left(\sum_{j=0}^{n}\wtilde_j \right)U_h^{0} - \sum^n_{j=1}\wtilde_j U_h^{n-j} +  \tau^{\alpha}P_h g(U_h^n).
		\label{fully_disc}
	\end{equation}

	Observe that, since $(A_h w , w) > 0$ for all $w \in X_h$, it is true that 
	
	$$\|(I  +  \tau^{\alpha}A_h)^{-1}\|_{L^2(\W)} \leq 1,$$
	for all $\tau>0$.  
	Now, suppose by induction, that we have a solution $U^m_h \in X_h$ for all $m < n$, and define $T: X_h \to X_h$ as
	
	\begin{equation} 
		T(w)  = (I  +  \tau^{\alpha}A_h)^{-1} \Big( \left(\sum_{j=0}^{n}\wtilde_j \right)U_h^{0} - \sum^n_{j=1}\wtilde_j U_h^{n-j} +  \tau^{\alpha}P_h g(w) \Big).
		\label{fully_disc2}
	\end{equation}
	Applying a fixed point argument, if $T$ is a contraction over $X_h$, then problem \eqref{fully_AC_rl} will have a unique solution. To this end, suppose that we have $u$ and $w \in X_h$. Then, using $|\fmonio'|<B$ we have
	
	$$\|T(u)-T(w)\|_{\ele} = \|(I  +  \tau^{\alpha}A_h)^{-1} \Big(\tau^{\alpha}P_h(\fmonio(u)-\fmonio(w)) \Big)\|_{\ele} $$ 
	$$
	\leq \tau^{\alpha}\| \fmonio(u) - \fmonio(w)\|_{\ele} \leq B\tau^{\alpha}\|u-w\|_{\ele}. 
	$$
	Taking $\tau < B^{-\alpha}$, we have that $T$ is a contraction, and problem \eqref{fully_AC_rl} has a unique solution.  \end{proof}

\section{Error estimation} \label{sec:error}

For the sake of simplicity we are going to consider $\eps^2 = 1$ through this section.

\subsection{Error estimates for the semidiscrete scheme}
For the following linear problem
\begin{equation}
	\label{eq:weaklineal}
	\left\lbrace
	\begin{array}{rl}      
		(\ppa u,\phii) + \eps^2\langle u , \phii \rangle_{H^s(\rn)}  & = (f,\phii) \quad \forall \phii \in \hsT, \\
		u(0)  & = v  \text{ in } \Omega, 
	\end{array}
	\right.
\end{equation}
where $f:[0,T]\to L^2(\Omega)$ and the corresponding
semi-discrete approximation
\begin{equation} 
	\label{eq:semiDiscretoLineal}
	\left\lbrace
	\begin{array}{rl}
		( \ppa u_h  , w )  +  \langle u_h, w \rangle_{H^s(\rn)} & =  \left( f, w \right), \quad \forall w \in X_h, \\
		u_h(0) & = v_h ,
	\end{array}
	\right.
\end{equation}
with $v_h = P_h v$, we have the following two results \cite{ABB2}.

\begin{theorem} \label{teo:semi}
	Let $u$ and $u_h$ be solutions of \eqref{eq:weaklineal}
	and \eqref{eq:semiDiscretoLineal} respectively with $v \in \dot{H}^{q}(\W)$, $q \in [0,2]$, and right hand side $f \equiv 0$. Then it holds, 
	$$
	\|u - u_h \|_{L^2(\W)} + h^{\gamma}|u - u_h |_{H^s(\rn)} \leq  C h^{2 \gamma} t^{-\alpha \big(\frac{2-q}{2} \big)} \|v\|_{\ele}.
	$$
	where  $C = C(s,n)$ and $\gamma = \min \{ s, 1/2 - \eps\}$, with $\eps > 0$ arbitrary small.
\end{theorem}

\begin{theorem}
	\label{estNoHom}
	Let $u$ and $u_h$ be as in Theorem \ref{teo:semi} with$f \in L^{\infty}( [0,T] ; L^2(\W) )$ and initial datum equal to zero. Then, there exists a positive constant $C = C(s,n)$ such that
	$$
	\|u - u_h \|_{L^2(\W)} \leq  C h^{2 \gamma } |\log h|^2 \|f\|_{ L^{\infty}( [0,T] ; L^2(\W) )},
	$$ 
	with $\gamma$ as in Theorem \ref{teo:semi}.
\end{theorem}

From this, we can estimate the error for the semi-discrete scheme \eqref{semiDiscreto}. 

\begin{theorem}
	\label{teo:errorAC}
	Let $u$ and $u_h$ be the the exact and the semi-discrete solution of \eqref{eq:weakAC} and \eqref{semiDiscreto} respectively. And let $v \in \dot{H}^{q}(\W)$ with $q \in [0,2]$ and $v_h = P_hv$ with $\|v\|_{q,s}\leq R$. Then there exist a positive constant $C=C( R , T )$ such that   
	
	\begin{equation}
		\|u(t) - u_h(t)\|_{L^2(\Omega)} \leq Ch^{2\gamma}(t^{-\alpha \big(\frac{2-q}{2} \big)} + |\log{h}|^2 ) \,\, , t\in (0,T].
		\label{eq:errorCAC}
	\end{equation}
	With $\gamma$ as in Theorem \ref{teo:semi}.
	
\end{theorem}

\begin{proof}
	
	We can write the solution and its semi-discrete approximation as 
	$u = E^{\alpha}(t)v + \int^{t}_{0} F^{\alpha}(t-s)\fmonio(u(s)) \, ds,$
	and
	$u_h = E^{\alpha}_h(t)v_h + \int^{t}_{0} F^{\alpha}_h(t-s)\fmonio(u_h(s)) \, ds,$
	respectively. Then, defining $e = u - u_h $, we have
	
	$$e(t) = \big(E^{\alpha}-E^{\alpha}_hP_h\big)(t)v + \int^{t}_{0} F^{\alpha}_h(t-s)P_h\big( \,\fmonio(u(s)) - \fmonio(u_h(s))  \,\big)\,ds$$
	
	$$+\int^{t}_{0} \big(F^{\alpha}-F^{\alpha}_hP_h\big)(t-s)\fmonio(u(s))\,ds. $$
	Using Theorem \ref{teo:semi} in the first term; $|\fmonio|,|\fmonio'| \leq B$, and \eqref{eq:semigEst2} in the second term; Theorem \ref{estNoHom} with $f = \fmonio(u)$ and $|\fmonio|<B$ in the last term, we have 
	
	$$\|e(t)\|_{L^2(\Omega)} \leq CRt^{-\alpha \big(\frac{2-q}{2} \big)}h^{2\gamma} + CB\int^{t}_0 (t-s)^{\alpha-1} \|e(s)\|_{L^2(\Omega)}\,ds + Ch^{2\gamma}|\log h|^2.$$
	Then, applying Lemma \ref{gronwall} we derive \eqref{eq:errorCAC}. \end{proof}

\subsection{Error estimation for the fully discrete scheme}

Consider the discrete problem of find $V^n_h \in X_h$, $n \in \{1,...,N\}$, $V^0_h = 0$ such that  

\begin{equation}
	\label{eq:duhamel_fuente}
	\sum^n_{j=0}\w_j V_h^{n-j}  = -  A_h V_h^n +  f^n_h, 
\end{equation}
with $f^n_h \in X_h$, for all $n \in \{1,...,N\}$. Recalling that $\w_0 = \tau^{-\alpha}$, and defining $E = (I + \tau^{\alpha} A_h)^{-1}$, we can rewrite \eqref{eq:duhamel_fuente} as

\begin{equation}
	\label{eq:duhamel_fuente2}
	V_h^n    = E \Big( \sum^n_{j=1}-\tau^{\alpha}\w_j V_h^{n-j}+  \tau^{\alpha}f^n_h \Big). 
\end{equation}
If we define $\{\wtilde_n\}_{n \in \N}$ as the coefficients of the series expansion of $(1 - \xi)^{\alpha}$, from the definition  of $\{\w_n\}_{n \in \N}$ we have $\wtilde_n = \tau^{\alpha}\w_n$ for all $n \in \N$. And we can write $V^{n}_h$ as a function of $f^{n}_h$ in a recursive expression 

\begin{equation}
	\label{eq:duhamel_fuente3}
	V_h^n    = \sum^{n}_{j = 1} E_{n - j}  f^j_h, \quad n>0, 
\end{equation}
with $E_n$ recursively defined as

\begin{equation}
	\label{eq:operador_duhamel_rec}
	E_{0} =  \tau^{\alpha}E, \quad E_n = E \Big( \sum^{n-1}_{j=0} -\wtilde_{n-j} E_{j} \Big).
\end{equation}


As we have observed in the proof of Theorem \ref{teo:solucion_fully_dis}, we have

$$\|E\|_{L^2(\W)} = \|(I + \tau^{\alpha} A_h)^{-1}\|_{L^2(\W)} < 1.$$Then, from \eqref{eq:operador_duhamel_rec}, and recalling that $-\wtilde_j>0$ for $j \geq 1$,  we have

\begin{equation}
	\label{eq:estimacion_duhamel}
	\|E_{0}\|_{L^2(\W)} \leq  \tau^{\alpha}, \quad \|E_n\|_{L^2(\W)} \leq  \sum^{n-1}_{j=0} -\wtilde_{n-j} \|E_{j}\|_{L^2(\W)}.
\end{equation}
Defining the sequence
\begin{equation}
	\label{eq:suc_importante}
	c_0 = 1, \quad c_n = \sum^{n-1}_{j=0} -\wtilde_{n-j} c_{j}, 
\end{equation}
it is possible to check that 

\begin{equation}
	\label{eq:est_operador_error}
	\|E_n\|_{L^2(\W)} \leq \tau^{\alpha} c_n.
\end{equation}  

In order to bound the error, it will be useful to know about the asymptotic behavior of $\{c_n\}_{n \in \N}$. This is analyzed in the next lemma proved in  Appendix \ref{appendix_AC}.

\begin{lemma}
	\label{lem:lemma_seq}
	Let $\{\wtilde_n\}_{n \in \mathbb{N}_0}$ be the coefficients of the power series expansion of $(1-\xi)^{\alpha}$, with $\alpha \in (0,1)$, and $\{c_n\}_{n \in \mathbb{N}_0}$ the sequence recursively defined in \eqref{eq:suc_importante}. Then, $c_n \in O(n^{\alpha - 1})$. 
	
\end{lemma}
\begin{theorem}
	\label{teo:error_fully_alfa}
	Let $u$ and $U^n_h = U_h(t_n)$ be the solution of \eqref{eq:AC} and \eqref{fully_AC_rl} respectively Consider $v \in \dot{H}^{q}(\W)$ for some $q \in (0,2]$ and $v_h = P_hv$ with $\|v\|_{q,s}\leq R$. Then, if $\tau < \tau_0$, for a sufficiently small $\tau_0$ there exist a positive constant $C=C( R , T , \alpha , q)$ such that   
	
	\begin{equation}
		\|u(t_n) - U_h(t_n)\|_{L^2(\Omega)} \leq Ch^{2\gamma}(t_n^{-\alpha \big(\frac{2-q}{2} \big)} + |\log{h}|^2 ) + C\tau t_n^{-1 + \alpha \frac{q}{2} }, 
		\label{eq:errorACF_alpha}
	\end{equation}
	$$ t_n\in (0,T].$$
	With $\gamma$ as in Theorem \ref{teo:semi}. 
\end{theorem}

\begin{proof}
	In view of Theorem \ref{teo:errorAC}, we only need to  estimate  $\|u_h(t_n) - U_h(t_n)\|_{L^2(\W)}$, with $u_h$ the semi-discrete solution. Considering the sector $\Sigma_{\theta} := \{ z \in \mathbb{C} \text{ such that } z \not = 0, |arg(z)| \leq \theta \}$, it can be seen that the function $G(z) := (z^{\alpha}I + A_h)^{-1}$ is analytic in $\Sigma_{\theta}$ with $\theta \in (\pi/2 , \pi)$. Then, from the semi-discrete and fully discrete scheme, we have
	
	$$u_h = G(\partial_t) \partial_t^{\alpha}v_h + G(\partial_t)P_h\fmonio(u_h), $$ 
	and 
	$$U_h = G(\overline{\partial_{\tau}}) \overline{\partial_{\tau}}^{\alpha}v_h + G(\overline{\partial_{\tau}})P_h\fmonio(U_h).$$
	Subtracting both expressions we obtain an equation for $e_h := u_h - U_h$, 
	
	\begin{equation}
		\label{eq:error_formula}
		e_h = (G(\partial_t) \partial_t^{\alpha} - G(\overline{\partial_{\tau}}) \overline{\partial_{\tau}}^{\alpha}P_h)v + G(\partial_t)P_h\fmonio(u_h)
		- G(\overline{\partial_{\tau}})P_h\fmonio(U_h) = 
	\end{equation}
	$$(G(\partial_t) \partial_t^{\alpha} - G(\overline{\partial_{\tau}}) \overline{\partial_{\tau}}^{\alpha}P_h)v_h + 
	(G(\partial_t) - G(\overline{\partial_{\tau}}) )P_h\fmonio(u_h) 
	+ G(\overline{\partial_{\tau}})P_h(\fmonio(u_h)-\fmonio(U_h))
	$$
	$$= (i) + (ii) + (iii).$$
	
	The norm of the first term $(i)$ can be estimated arguing as in \cite[Theorem 5.3]{ABB2} (see also \cite[Theorem 4.2.8]{tesis_yo}). Since $v \in \dot{H}^{q}(\W)$ we obtain
	
	$$\|(i)\|_{L^2(\W)} \leq Ct_n^{-1 + \alpha \frac{q}{2}}\tau\|v_h\|_{H^q(\W)} \leq Ct_n^{-1 + \alpha \frac{q}{2}}\tau,$$
	with $C = C(R)$. For the second term, using property \eqref{eq:asociatividad}, we can split $(ii)$ as follow 
	
	$$(ii) = \big(G(\partial_t) - G(\overline{\partial_{\tau}}) \big)\big(\, P_h\fmonio(u_h(0)) + (1*P_h\partial_t\fmonio(u_h(t_n))\,\big)  $$
	$$
	= \big(G(\partial_t) - G(\overline{\partial_{\tau}}) \big)P_h\fmonio(u_h(0)) + (\big(G(\partial_t) - G(\overline{\partial_{\tau}}) \big)1)*P_h\partial_t\fmonio(u_h(t_n))
	$$
	$$
	= I + II.
	$$
	Using Lemma \ref{lem:est_lubich} with $\mu = \alpha, \beta =1$, along with the fact that $|\fmonio| < B$, we can estimate 
	
	$$\|I\|_{\ele} \leq Ct_n^{\alpha-1}\tau.$$
	
	On the other hand, noticing that Lemma \ref{lem:estimacion_derivada} can be easily extended to $\partial_t u_h$, in order to get $\|\partial_t u_h\|_{\ele} \leq C t^{\alpha/2 -1}$; using again Lemma \ref{lem:est_lubich}, the fact that $|\fmonio'|<B$, and writing $\partial_t \fmonio(u_h(t)) = \fmonio'(u_h(t))\partial_t u_h$, we have 
	
	$$
	\|II\|_{\ele} \leq \int^{t_n}_{0} \|(\big(G(\partial_t) - G(\overline{\partial_{\tau}}) \big)1)(t_n - s)\fmonio'(u_h(s))\partial_t u_h(s)\|_{\ele} \, ds
	$$
	
	$$ \leq  C\tau\int^{t_n}_0 (t_n - s)^{\alpha-1}  \|\partial_t u_h(s)\|_{L^2(\Omega)} \,ds \leq C\tau\int^{t_n}_0 (t_n - s)^{\alpha-1} s^{\alpha/2 -1} \,ds
	$$
	$$
	\leq C \tau t_n^{\frac{3}{2}\alpha -1},
	$$
	where in the last inequality we have estimated the integral in terms of the beta function $B(\alpha/2,\alpha)$, as in Theorem \ref{EUlocal}.

	Now, we observe that the last term $(iii)$ is a solution for \eqref{eq:duhamel_fuente}, with $f^n_h = P_h(\fmonio(u_h)-\fmonio(U_h))$. Then, in view of \eqref{eq:duhamel_fuente3} and \eqref{eq:est_operador_error}, and using again that $|\fmonio'|<B$, we have 
	
	$$\|(iii)\|_{L^2(\W)} \leq \tau^{\alpha}\sum^{n}_{j=1} c_{n - j} \|e_h(t_j)\|_{L^2(\W)}, $$
	where $\{c_n\}_{n \in \N}$ is the sequence defined in \eqref{eq:suc_importante}.
	
	Using that $C\tau(  t_n^{-1 + \alpha \frac{q}{2}} + t_n^{\frac{3}{2}\alpha -1} +  t_n^{\alpha -1} ) \leq C\tau t_n^{-1 + \alpha \frac{q}{2}},$
	with $C = C(T)$, and the fact that $\tau^{\alpha}c_n \sim \tau^{\alpha}(n+1)^{\alpha-1} = \tau t^{\alpha-1}_{n+1}$ (given by Lemma \ref{lem:lemma_seq}), we can derive the following  
	
	\begin{equation*}
		\label{eq:est_error_full_alfa}
		\|e_h(t_n)\|_{L^2(\W)} \leq C \tau t_n^{-1 + \alpha \frac{q}{2}} + C\tau\sum^{n}_{j=1}  t^{\alpha-1}_{n-j+1} \|e_h(t_j)\|_{L^2(\W)} = 
	\end{equation*} 
	
	$$C \tau t_n^{-1 + \alpha \frac{q}{2}} + C\tau\sum^{n-1}_{j=1}  t^{\alpha-1}_{n-j+1} \|e_h(t_j)\|_{L^2(\W)} + C\tau^{\alpha} \|e_h(t_n)\|_{L^2(\W)}.$$
	Taking $ \tau_0$ in such a way that $1 - \tau^{\alpha}_0 C > 0 $ and $\tau < \tau_0$, we can subtract the last term on the right from both sides an obtain
	
	\begin{equation*}
		\|e_h(t_n)\|_{L^2(\W)} \leq \frac{C}{1-\tau^{\alpha}_0C} \tau t_n^{-1 + \alpha \frac{q}{2}} + \frac{C}{1-\tau^{\alpha}_0C}\tau\sum^{n-1}_{j=1}  t^{\alpha-1}_{n-j+1} \|e_h(t_j)\|_{L^2(\W)}.
	\end{equation*} 
	Finally, applying Lemma \ref{lem:gronwall_discreto} (a discrete analog of Lemma \ref{gronwall}), we have 
	
	\begin{equation}
		\label{eq:est_error_full_alfa2}
		\|e_h(t_n)\|_{L^2(\W)} \leq C \tau t_n^{-1 + \alpha \frac{q}{2}} ,
	\end{equation} 
	for some $C = C(R,T,\alpha,q)$.

	From this, \eqref{eq:est_error_full_alfa2}, and \eqref{eq:errorCAC}, we can derive \eqref{eq:errorACF_alpha}.    \end{proof}

\section{Analysis of the fractional Allen-Cahn equation} 
\label{sec:analisisAC}
\subsection{ $L^{\infty}$ bounds}
\label{sec:bounds}
At this point we need to recall some results that play an important role in our analysis.  The following two theorems summarize classical global and interior regularity of solutions for the following problem, 
\begin{equation}
	\left\lbrace
	\begin{array}{rl}      
		(-\Delta)^s u  & = f  \text{ in } \Omega \\
		u & =  0  \text{ in }    \mathbb{R}^n \setminus \Omega.
	\end{array}
	\right.
	\label{eq:elliptic}
\end{equation}
We refer to \cite{FR} for further details about \eqref{eq:elliptic}.  

\begin{theorem}
	\label{teo:elliptic_reg2}
	Let $\W \subset \mathbb{R}^n$ be any bounded $C^{1,1}$ domain, $s \in (0,1)$, and $u$ be the solution of \eqref{eq:elliptic}. If $f \in L^{\infty}(\W)$; then $u \in C^s(\mathbb{R}^n)$. Moreover, $\|u\|_{C^s(\mathbb{R}^n)} \leq C \|f\|_{L^{\infty}(\W)},$
	where the constant $C$ depends only on $\W$ and $s$.  
\end{theorem}

\begin{theorem}
	\label{teo:elliptic_reg}
	Let $\W$ be a bounded domain of $\mathbb{R}^n$, and let $u$ be a solution for \eqref{eq:elliptic}. If $\delta(x)= dist(x,\partial \W)$, for each $\rho > 0$ define $\W_{\rho} := \{ x \in \W : \delta(x) > \rho \}$. Then, if $\beta + 2s$ is not an integer, for every $0<\rho'<\rho$ we have 
	
	\begin{equation}
		\label{eq:reg_est}
		\|u\|_{C^{\beta + 2s}(\W_{\rho})} \leq C \| f \|_{C^{\beta}(\W_{\rho'})},
	\end{equation}
	with $C = C(n,s,\W,\beta,\rho,\rho')$.

\end{theorem}

We devote the remaining sections to problem \eqref{eq:AC}. In order to exploit the results obtained for \eqref{eq:AC_trunc} we need to work with a truncated version of $f$. In this way, along this section we assume that our source term $g$ verifies 
\eqref{H1}, \eqref{H2} and agrees with $f$ in an interval  $[-1 - R , 1 + R]$, for some $R>0$. In order to prove that, in this case, the solution remains bounded between $1$ and $-1$ for any $v$ such that $\|v\|_{\infty}\le 1$, we are going to define first a discrete in time problem. That is, find $U^n \in \hsT$, with $n \in \{1,...,N\}$, such that 

\begin{equation} 
	\left\lbrace
	\begin{array}{rl}
		\dtd^{\alpha} U^n   +  A U^n & =  \dtd^{\alpha}v +  \fmonio(U^n) \\
		U^0 & = v.
	\end{array}
	\label{eq:semi_in_time}
	\right.
\end{equation}

The proof of the existence and uniqueness of solutions for this problem is similar to the one given for the fully discrete case. For the solution of this problem, we have the following result. 

\begin{theorem}
	\label{teo:existencia_bound_semi_tiempo}
	Consider the semi-discrete in time scheme \eqref{eq:semi_in_time} with $U^0 \in L^{\infty}(\W)$, then there exist $\tau_0>0$ in such a way that if $\tau < \tau_0$ \eqref{eq:semi_in_time} has a solution $U^n$, $n \in \{0,...,N\}$, with $U^n \in C^{s}(\mathbb{R}^n)$ for all $n>0$. Moreover if $|U^0(x)| \leq 1$ for all $x \in \W$, then $|U^n(x)| \leq 1$ for all $x \in \W$ and $n \in \{1,...N\}$. 
	
\end{theorem}

\begin{proof}
	
	Suppose we have a solution with the desire properties for all $m<n$. From \eqref{eq:semi_in_time} we have the identity 
	
	\begin{equation} 
		U^n  =  ( I  +  \tau^{\alpha}  A)^{-1} \Big( \left(\sum_{j=0}^{n}\wtilde_j \right)U^{0} -  \sum^n_{j=1}\wtilde_j U^{n-j} + \tau^{\alpha} \fmonio(U^n) \Big), 
		\label{eq:semi_time_ex1}
	\end{equation}
	where $\wtilde_j := \tau^{\alpha}\w_j$. 
	
	First we want to show that there exists $U^n \in \ele$ that satisfies equation \ref{eq:semi_time_ex1}. In order to do that, we define the map $T: L^{2}(\W) \to L^{2}(\W)$
	
	$$T(u) :=  ( I  +  \tau^{\alpha}  A)^{-1} \Big( \left(\sum_{j=0}^{n}\w_j \right)U^{0} -  \sum^n_{j=1}\w_j U^{n-j} +  \tau^{\alpha}\fmonio(u) \Big).$$
	
	We want to verify that $T$ is a contraction in $L^{2}(\W)$. From the fact that $A$ is a maximal monotone operator (see \cite{Brezis}), we know that $\|(I + \tau^{\alpha}A)\|_{\ele} \leq 1$. Let $u$ and $v$ $\in L^{2}(\W)$, we can estimate 
	
	$$\| T(u) - T(v)  \|_{L^{2}(\W)} = \|  ( I  +  \tau^{\alpha}  A)^{-1} \tau^{\alpha} \big( \fmonio(u) - \fmonio(v) \big)  \|_{L^{2}(\W)} $$
	
	$$ \leq \tau^{\alpha} \| \fmonio(u) - \fmonio(v) \|_{\ele} \leq \tau^{\alpha}B\| u - v \|_{\ele}. $$
	Then, for a small $\tau$ we have that $T$ is a contraction. Hence, there exists a unique solution $U^n \in \ele$ for \eqref{eq:semi_time_ex1}, and the identity
	
	\begin{equation}
		\label{eq:semi_time_ex2} 
		A U^n  = \left(\sum_{j=0}^{n}\w_j \right)U^{0} -  \sum^n_{j=1}\w_j U^{n-j} +  \fmonio(U^n) 
	\end{equation}
	is satisfied. Since the right hand side belongs to $L^{\infty}(\W)$, applying the Theorem \ref{teo:elliptic_reg2}, we can conclude that $U^n \in C^s(\mathbb{R}^n) \cap C^{2s}(\W_{\rho})$ for all $0<\rho<\rho_0$. 
	
	Now, we want to see that if the initial data is regular enough, then the solution remains bounded between $1$ and $-1$. Indeed, suppose we have $U^m \in  C^{2}(\W) \cap C^s(\mathbb{R}^n) $ and $|U^m(x)|\leq 1$ for all $x \in \W$, for all $m<n$. If we take a fixed $\rho'>0$ with $\rho=2\rho'$ in Theorem \ref{teo:elliptic_reg}, and use the the fact that $U_n \in C^{2s}(\Omega_{\rho})$, then $g(U_n) \in C^{2s}(\Omega_{\rho})$ and we can conclude that $U_n \in C^{2s + 2s}(\Omega_{\rho'})$. A repeated application of this argument, along with the fact that $\fmonio \in C^2(\mathbb{R})$, implies that $U^n \in C^{2+2s}(\W_{k\rho_0})$ for some $k \in \N,$ only depending on $s$. Since $\rho_0$ can be arbitrary small, we can assert that $U^n \in  C^{2}(\W)$, and then, $U^n \in  C^{2}(\W) \cap C^s(\mathbb{R}^n)$.

	On the other hand, the semi-discrete in time scheme gives us the relation           
	
	\begin{equation*}
		\label{eq:semi_time_ex3}
		\sum^n_{j=0}\w_j U^{n-j} - \left(\sum_{j=0}^{n}\w_j \right)U^{0}  = -  A U^n +  \fmonio(U^n), 
	\end{equation*}
	which can be rewritten as 
	
	\begin{equation*}
		\label{eq:semi_time_ex4}
		\sum^{n-1}_{j=0}a_j (U^{n-j}-U^{n-j-1})   = - A U^n +  \fmonio(U^n), 
	\end{equation*}
	with $a_n = \sum^{n}_{j=0} \w_j$.  
	Suppose that there exist some $x_0$ such that $U^n$ achieves its maximum on that point, and $U^n(x_0)>1$. Recall that $\|U^m\|_{L^{\infty}(\W)} \leq 1$ for all $m<n$. From the regularity of $U^n$, it can be shown that $A U^n(x_0) = (-\Delta)^s U^n(x_0) \geq 0$ (see \cite[Lemma 3.9]{FR}). Then, from the fact that $U^n(x_0)>1$, we have $\fmonio(U^n(x_0))<0$, which implies
	
	\begin{equation*}
		\label{eq:semi_time_ex5}
		\sum^{n-1}_{j=0}a_j (U^{n-j}(x_0)-U^{n-j-1}(x_0))   < 0.
	\end{equation*}
	
	Observing the fact that $\{a_n\}$ is a positive and strictly decreasing sequence, it is possible to show that there exist $m_0<n$, such that $U^{m_0}(x_0) > U^{n}(x_0)$ (see \cite[Lemma 5.2.4]{tesis_yo}), and then $1 \geq U^{m_0}(x_0) > U^{n}(x_0) > 1$. The contradiction came from the assumption that $U^{n}(x_0) > 1$.
	
	Now we want to see that the same bound holds for less regular initial data. To this end, applying a density argument, suppose $U^n$ is a solution for \eqref{eq:semi_in_time} with  $U^0 \in L^{\infty}(\W)$, $\|U^0\|_{L^{\infty}(\W)}\leq 1$. Consider $\{U^0_k\}_{k \in \N} \subset C^{\infty}_c(\W)$, with $\|U^0_k\|_{L^{\infty}(\W)}\leq 1$ for all $k$, and $ U^0_k \to U^0 $ in $\ele$. 
	
	Let $U^n_k$ be the solution of \eqref{eq:semi_in_time} with initial data $U^0_k$. Calling $e^n_k = U^n - U^n_k$, we have the equation
	
	\begin{equation*} 
		e^n_k  =  ( I  +  \tau^{\alpha}  A)^{-1} \Big( \left(\sum_{j=0}^{n}\wtilde_j \right)e^{0}_k -  \sum^n_{j=1}\wtilde_j e^{n-j}_k +  \tau^{\alpha} \big( \fmonio(U^n) - \fmonio(U^n_k) \big) \Big),
		\label{eq:semi_time_ex6}
	\end{equation*}
	and taking norms we obtain 
	
	\begin{equation} 
		\| e^n_k \|_{\ele}  \leq  \|e^{0}_k\|_{\ele} +  \sum^n_{j=1}-\wtilde_j \| e^{n-j}_k\|_{\ele} +  \tau^{\alpha} B \| e^n_k \|_{\ele}
		\label{eq:semi_time_ex7}.
	\end{equation}
	Choosing $\tau_0$ such that $\tau_0^{\alpha}B<1$, recalling that $0 < -\wtilde_j \leq -\wtilde_1$, and applying a discrete Gronwall type inequality we have
	
	$$\| e^n_k \|_{\ele} \leq \|e^{0}_k\|_{\ele} \frac{e^{n\wtilde_1/1 - \tau_0^{\alpha}B}}{1 - \tau_0^{\alpha}B} = C\|e^{0}_k\|_{\ele}, $$
	with $C=C(\tau_0,\alpha,B,n)$, and then, $\| e^n_k \|_{\ele} \to 0$ with $k \to \infty$.   
	
	Since $\| U^0_k \|_{L^{\infty}(\W)} \leq 1$, then $\| U^n_k \|_{L^{\infty}(\W)} \leq 1$ for all $k$, and for all $n \in \{1,...,N\}$. Hence, for a fixed $n$, we can construct a sub-sequence $\{U^n_{k_j}\}_{j\in \N}$, such that $U^n_{k_j} \to U^n$ a.e. , and conclude that $\|U^n\|_{L^{\infty}(\W)} \leq 1$.      \end{proof}

Finally, proceeding analogously as in Theorem \ref{teo:error_fully_alfa}, we can derive the following error estimation. 

\begin{theorem}
	\label{teo:error_semi_tiempo}
	Let $u$ and $U^n = U(t_n)$ be the solution of \eqref{eq:AC} and \eqref{eq:semi_in_time} respectively, with $v \in \dot{H}^{q}(\W)$, $\|v\|_{q,s}\leq R$ with $q \in (0,2]$. Then, if $\tau < \tau_0$, for a sufficiently small $\tau_0>0$ exists a positive constant $C=C( R , T , \alpha , q)$ such that   
	
	\begin{equation}
		\|u(t_n) - U(t_n)\|_{L^2(\Omega)} \leq C t_n^{ - 1 + \alpha \frac{q}{2}}\tau, \quad t_n\in [0,T]. 
		\label{eq:error_semi_tiempo}
	\end{equation}
\end{theorem}

Now, consider $\|v\|_{L^{\infty}(\W)} \leq 1$. Given a fixed $t \in (0,T]$ we can construct a family of nested partitions of [0,T] with $\tau = T/N_k$, $k \in \N$, and $N_k \to \infty$ if $k \to \infty$, in such a way that $t$ belongs to all the partitions. Let $U_k$ be the solution of \eqref{eq:semi_in_time}, and $u$ the solution of \eqref{eq:weakAC}, using Theorem \ref{teo:error_semi_tiempo} we have that $U_k(t) \to u(t)$ in $\ele$. So, we can extract a subsequence $\{U_{k_j}(t)\}_{j \in \N}$ such that $U_{k_j}(t) \to u(t)$ a.e. . Using \ref{teo:existencia_bound_semi_tiempo}, we know that $\|U_k(t) \|_{L^{\infty}(\W)} \leq 1$, and then, $\|u(t) \|_{L^{\infty}(\W)} \leq 1$. We can summarize this observation in the following result.

\begin{theorem}
	Let $u$ a solution of \eqref{eq:weakAC2} with $\|v\|_{L^{\infty}(\W)} \leq 1$. Then $\|u(t)\|_{L^{\infty}(\W)} \leq 1$ for all $t \in (0,T]$.  
\end{theorem}

This theorem implies that all the analysis displayed up to here remains valid replacing $\fmonio$ by $f$ and therefore to the Allen-Cahn equation \eqref{eq:AC}.

\section{Discussion about the asymptotic behavior with $s \rightarrow 0$} 
\label{sec:asymptotic}

Considering now the usual derivative in time ($\alpha = 1$), the Allen-Cahn equation can be understood as a gradient flow in $L^2$, minimizing the free energy functional   

\begin{equation}
	F_s(u) = \frac{\e^2}{2}|u|^2_{H^s(\rn)} + \int_{\Omega} W(u),
	\label{eq:funcional_free}
\end{equation} 
with $W(u) = \frac{(u^2 - 1)^2}{4}$ (see for example \cite{ains}). 
It is well known that the size of $\eps$ affects the interface width of the minimizers of $F_s$. That is, interface width tends to zero with $\eps \to 0$. This fact can be easily derived from expression \eqref{eq:funcional_free}, observing that the right term, which penalizes the variation of $u$, tends to lose relevance as $\eps$ goes to zero, forcing the minimizer $u$ to take values into the set of minimizers of $W$, that is values belonging to $\{1 , -1 \}$. However, since $\eps > 0$, the right term promote the minimization of the interface length (for $n \geq 2$), which implies that the limit behavior cannot be understood as the minimization of $F_s$ with $\eps = 0$. In \cite{SV}, Savin and Valdinoci show, by means of $\Gamma$-convergence theory, that the limit behavior of the problem of minimizing  $F_s$ tends to a minimal surface problem if $s \in [1/2,1)$, and to a non-local version of the minimal surface problem for $s \in (0,1/2)$.

In our case, numerical experiments (see Figure \ref{ejemplo_s}) show that the interface width tends to become thinner when the parameter $s$ goes to zero, suggesting that (as in the case $\eps \to 0$) a minimizer of $F_s$ should approximate a binary function when $s \to 0$. This behavior was also observed in \cite{ains2}, where authors derive the scaling law $\tau = O(\eps^{1/2s})$, with $\tau$ denoting the interface width. On the other hand, a displacement of the equilibrium states have been observed in our numerical experiments for small values of the fractional parameter $s$ (see Figure \ref{ejemplo_1}).


Motivated by the previous observation, the aim of this section is to analyze the asymptotic behavior of the minimizers of $F_s$ with $s$ tending to zero. To this end, we are going to follow the ideas displayed in \cite{SV}, and study the $\Gamma$-convergence of a suitable modification of the functional $F_s$. By means of this framework, it is possible to conclude that the minimizers of \eqref{eq:funcional_free} approach binary functions, and the equilibrium states should be placed near $\sqrt{(1 - \eps^2)}$ and $-\sqrt{(1 - \eps^2)}$ for small values of the fractional parameter $s$.

\subsection{$\Gamma$-convergence when $s \rightarrow 0$.}

Since $\Gamma$-convergence may not be a usual concept in numerical analysis, we start this section by giving its definition and basic properties, and we refer to \cite{gamma} for further details. 

Let $X$ be a topological space, and $\{F_n\}_{\in \N}$, $F_n : X \to [-\infty,+\infty]$, a sequence of functionals. Then, we say that $F_n$ $\Gamma$-converge to $F: X \to [-\infty,+\infty]$, if the following conditions holds: 

\begin{itemize}
	\item For every sequence $\{x_n\}_{n \in \N} \subset X$ such that $x_n \to x$, then 
	$$F(x) \leq \liminf_{n \to \infty} F_n (x_n).$$  
	
	\item For every $x \in X$, there exist a sequence $x_n$ converging to $x$ such that
	$$F(x) \geq \limsup_{n \to \infty} F_n (x_n).$$
\end{itemize} 

Also, we define a complementary concept. We say that the family  $\{F_n\}$ has the equi-coerciveness property if for all $c \in \mathbb{R}$ exists a compact set $K_c$ in such a way that $\{F_n < c \}\subset K_c$ for all $n \in \N$.

These two concept allow us to say something about the limiting behavior of the minimizers of $F_n$ in terms of the minimizers of $F$. That is, if $x_n$ is a minimizer of $F_n$, then every cluster point of $\{x_n\}_{n \in \N}$ (if exist) is a minimizer of $F$. This can be summarized as follow 

$$\text{ Equi-coerciveness} + \Gamma\text{-convergence} \Rightarrow \text{Convergence of minimizers}$$


In order to study the $\Gamma$-convergence of $F_s$, we must set an appropriate domain $X$ for $F_s$, 

$$X=\{u \in L^{\infty}(\rn) \text{ with } |u|\leq 1, \text{ and } u \equiv 0 \text{ in } \Omega^c \}.$$ 
And we are going to consider this space furnished with the norm $\| \cdot \|_{L^1(\W)}$. Note that if $u \in X$ but $u \not \in \tilde{H}^s(\Omega)$, then we can define $F_s(u) = +\infty$.

From the definition of $F_s$, and supposing $\e^2<1$, we have 

$$F_s(u) = \frac{\e^2}{2}|u|^2_{H^s(\rn)} - \frac{\e^2}{2}\|u\|^2_{\ele} + \frac{\e^2}{2}\|u\|^2_{\ele} + \int_{\Omega} W(u) $$

$$= \frac{\e^2}{2}\big( |u|^2_{H^s(\rn)} - \|u\|^2_{\ele}\big) + \int_{\Omega} \big( W(u) +\frac{\e^2}{2}u^2 \big), $$
and, denoting $\Fu[u](\xi)$ as the Fourier transform of $u$, we know from \cite{Hitchhikers} and Plancharel's identity that
$|u|^2_{H^s(\rn)} = \|\mathcal{F}[u]|\xi|^s\|^2_{L^2(\rn)} = \int_{\rn} \mathcal{F}^2[u](\xi)|\xi|^{2s} \, d\xi, $
and
$\|u\|^2_{\ele} = \int_{\rn}\mathcal{F}^2[u](\xi)\, d\xi .$
Then we have 

$$|u|^2_{H^s(\rn)} - \|u\|^2_{\ele} = \int_{\rn} \big(|\xi|^{2s} - 1\big)\Fu^2[u](\xi) \, d\xi, $$
so we can rewrite $F_s$ as

$$F_s = \frac{\e^2}{2}\int_{\rn} \big(|\xi|^{2s} - 1\big)\Fu^2[u](\xi) \, d\xi + \int_{\Omega} \tilde{W}(u), $$
with $\tilde{W}(s) =  W(s) +\frac{\e^2}{2}s$. 

Since we have $\e^2<1$, $\tilde{W}(s)$ is a double-well type potential with minimizers $\pm\sqrt{1-\e^2}$.

Noticing that $\tilde{W}(\pm\sqrt{1-\e^2}) = k_{\e}>0$ , we define a new auxiliary functional $\tilde{F}_s$ 

$$\tilde{F}_s  = \frac{1}{2s}\big(F_s - \int_{\Omega}k_{\e} \big),$$
and, for the sake of simplicity, we redefine $\tilde{W}$ as $\tilde{W}(s) =  W(s) +\frac{\e^2}{2}s - k_{\e},$
so now $\tilde{W}(\pm\sqrt{1-\e^2})=0$. 
Fixing $s$ and $\e$, it is easy to check that $u \in X$ is a minimizer of $F_s$ if and only if $u$ is a minimizer of $\tilde{F}_s$. So we focus our study on the asymptotic behavior of $\tilde{F}_s$.     

Defining the functional

\begin{equation}
	F_0(u) = \left\lbrace
	\begin{array}{rl}      
		\int_{\rn} \ln{|\xi|} \Fu^2[u](\xi) \, d\xi,  & \text{ if } u \equiv \sqrt{1-\e^2}\big(I_{E} - I_{E^c} \big) \\
		+\infty,  & \text{in other case},
	\end{array}
	\right.
	\label{eq:funcional}
\end{equation}
with $E \subset \Omega$, we have the following theorem. 

\begin{theorem}
	Let $\tilde{F}_s$ and $F_0$ defined as before, then $\tilde{F}_s \xrightarrow[]{\Gamma} F_0$. 
\end{theorem}

\begin{proof}
	Let $u_s \xrightarrow[]{} u$ with $s \rightarrow 0$ in $X$, and suppose w.l.o.g, that $s$ takes values in a discrete set. First, we want to see 
	\begin{equation}
		\liminf_{s \rightarrow 0} \tilde{F}_s(u_s) \geq F_0(u).
		\label{eq:liminff}
	\end{equation}  
	Indeed, suppose that $l = \liminf_{s \to 0}{\tilde{F}_s(u_s)} \leq +\infty$, in other case there is nothing to prove. If we choose a suitable sub-sequence of $u_s$ such that $u_s \to u$ a.e. and $\tilde{F}_s(u_s) \to l$, then
	
	\begin{equation}
		l = \liminf_{s \to 0}{ \tilde{F}_s(u_s) } \geq \liminf_{s \to 0}{ \int_{\rn} \frac{|\xi|^{2s}-1}{2s}} \Fu^2[u_s](\xi) \, d\xi + \liminf_{s \to 0}{\frac{1}{2s}\int_{\Omega} \tilde{W}(u_s)}
		\label{eq:liminf}
	\end{equation}   
	
	We first analyze the left term of the right hand side of \eqref{eq:liminf}.
	In this case we have  
	
	$$ \liminf_{s \to 0}{ \int_{\rn} \frac{|\xi|^{2s}-1}{2s}} \Fu^2[u_s](\xi) \, d\xi \geq \liminf_{s \to 0}{ \int_{|\xi|>1} \frac{|\xi|^{2s}-1}{2s}} \Fu^2[u_s](\xi) \, d\xi $$ 
	$$ + \liminf_{s \to 0}{ \int_{|\xi|\leq1} \frac{|\xi|^{2s}-1}{2s}} \Fu^2[u_s](\xi) \, d\xi = (i) + (ii).$$
	
	From the fact that $u_s \to u$ in $L^1(\rn)$ norm, we have $\Fu[u_s] \to \Fu[u]$ point-wise, and we also have $(|\xi|^{2s}-1)/2s \to \ln{|\xi|}$. Then, using Fatou's Lemma, we get the estimation
	
	$$(i) \geq   \int_{|\xi|>1} \ln{|\xi|}\Fu^2[u](\xi) \, d\xi >0 $$ 
	On the other hand, since $|\Fu[u_s](\xi)| \leq \|u_s\|_{L^1(\Omega)} \leq |\W|$, we can estimate the second term as follow
	
	$$(ii) = - \limsup  \int_{|\xi|\leq1} \frac{1-|\xi|^{2s}}{2s} \Fu^2[u_s](\xi) \, d\xi \geq - \int_{|\xi|\leq1} -\ln{|\xi|} \Fu^2[u](\xi) \, d\xi$$
	\begin{equation}
		=\int_{|\xi|\leq1} \ln{|\xi|} \Fu^2[u](\xi) \, d\xi > -\infty,  
		\label{eq:est}
	\end{equation}
	where in the last inequality we have use the reverse Fatous's Lemma.

	Hence, the first term on the right hand side of \eqref{eq:liminf} must be a finite number. This implies that  
	$0 \leq \liminf_{s \to 0} \frac{1}{2s}\int_{\Omega} \tilde{W}(u_s) < +\infty, $
	and thus, $\int_{\Omega} \tilde{W}(u_s) \to 0 $ with $s \to 0$. 
	
	Since we have chosen $u_s$ in such a way that $u_s \to u$ a.e., we have that $\tilde{W}(u) = 0$ a.e., then $u$ must have the form $u = \sqrt{1 - \e^2}(I_{E} - I_{E^c}). $   
	Now, we can estimate 
	
	$$\liminf_{s \rightarrow 0} \tilde{F}_s(u_s) = \liminf_{s \to 0}{ \int_{\rn} \frac{|\xi|^{2s}-1}{2s}} \Fu^2[u_s](\xi) \, d\xi + {\frac{1}{2s}\int_{\Omega} \tilde{W}(u_s)}$$   
	
	$$\geq \int_{\rn} \ln{|\xi|} \Fu^2[u](\xi) \, d\xi  = F_0(u),$$
	and \eqref{eq:liminff} follow.
	
	Finally we only need the to verify that if $u \in X$, then 
	
	\begin{equation}
		F_0(u) \geq \limsup_{s \to 0} F_s(u)
		\label{eq:limsup}	
	\end{equation}
	To this end, suppose $u = \sqrt{1-\e^2}(I_{E} - I_{E^c})$, otherwise there is nothing to prove. In this case we have $F_s(u) = \int_{\rn} \frac{|\xi|^{2s} - 1}{2s}\Fu^2[u](\xi) \, d\xi. $
	The fact that $(|\xi|^{2s} - 1)/2s \searrow \ln{|\xi|}$ with $s \to 0$, implies that $\tilde{F}_s(u)$ is decreasing in $s$. Then
	
	$$ \limsup_{s \to 0} \tilde{F}_s(u) =   \lim_{s \to 0} \tilde{F}_s(u) = \lim_{s \to 0} \int_{\rn} \frac{|\xi|^{2s} - 1}{2s}\Fu^2[u](\xi) \, d\xi $$ 
	$$= \lim_{s \to 0} \Big( \int_{|\xi|\leq 1} + \int_{|\xi|> 1} \Big) \frac{|\xi|^{2s} - 1}{2s}\Fu^2[u](\xi) \, d\xi. $$
	Then, using Monotone Convergence Theorem on the integral over $|\xi|>1$, and Dominated Convergence Theorem over $|\xi| \leq 1$, we have $\limsup_{s \to 0} \tilde{F}_s(u) = F_0(u), $
	which proves \eqref{eq:limsup}.

\end{proof}


\subsection{Equi-coerciveness of $\tilde{F}_s$} 

To complete the analysis we prove the equi-coerciveness of $\{\tilde{F}_s\}_s$.  

\begin{theorem}
	Suppose $s_n \to 0$, and $\{u_n\}_{n \in \N} \subset X$, such that $\tilde{F}_{s_n}(u_n) \leq C$ for all $n \in \N$. Then there exist $u \in X$ and a subsequence $\{u_{n_j}\}_{j \in \N}$, such that $u_{n_j} \to u$ in $X$.  
\end{theorem}  

\begin{proof}
	First we observe that, as before, we have
	
	$$ \tilde{F}_{s_n}(u_n) = \int_{\rn} \frac{|\xi|^{2s_n}-1}{2s_n}\Fu^2(u_n)[\xi] \, d\xi +  \frac{1}{2s_n}\int_{\W} \tilde{W}(u_n) \, dx.$$
	Since $(|\xi|^{2s} - 1)/2s \searrow \ln{|\xi|}$ and $\tilde{F}_{s_n}(u_n) \leq C$ for all $n \in \N$, we can assert that
	
	\begin{equation}
		\label{eq:cota_uniforme}
		\int_{\rn} \ln|\xi|\Fu^2(u_n)[\xi] \, d\xi \leq C, \quad \forall n \in \N.
	\end{equation}
	The fact that $u_{n} \in X$, implies $\|u_{n}\|_{L^1(\W)} \leq C_0$ for all $n \in \N$, and then 
	
	\begin{equation}
		\label{eq:cota_inf}
		\int_{|\xi|\leq 1} \ln|\xi|\Fu^2(u_n)[\xi] \, d\xi \geq C^2_0\int_{|\xi|\leq 1} \ln|\xi| \, d\xi \geq C_1.
	\end{equation}
	
	Now we want to show that \eqref{eq:cota_uniforme} implies that functions in the set $\{\Fu^2(u_n)\}_n$ keep a substantial part of their mass uniformly  bounded. Namely, given $\eta > 0$, there exist $R>0$ such that 
	
	\begin{equation}
		\label{eq:concentracion}
		\int_{|\xi|>R} \Fu^2(u_n)[\xi] \, d\xi \leq \eta, \quad \forall n \in \N.
	\end{equation} 
	By contradiction suppose that there exist $\eta_0 > 0$ such that for every $R>0$ there is a number $m = m(R,\eta_0) \in \N$, in such a way that 
	
	$$\int_{|\xi|>R} \Fu^2(u_m)[\xi] \, d\xi > \eta_0.$$
	Choosing $R>0$ such that $\eta_0 \ln (R) + C_1 > C$, with $C$ the constant in \eqref{eq:cota_uniforme}, and $C_1$ the one in \eqref{eq:cota_inf} , we have
	
	$$\int_{\rn} \ln|\xi|\Fu^2(u_m)[\xi] \, d\xi > \int_{|\xi|>R} \ln|\xi|\Fu^2(u_m)[\xi] \, d\xi + \int_{|\xi|\leq 1} \ln|\xi|\Fu^2(u_m)[\xi] \, d\xi $$ 
	$$ >  \ln(R)\eta_0 + C_1 > C,$$
	which, in view of \eqref{eq:cota_uniforme} results in a contradiction. Hence, assertion \eqref{eq:concentracion} holds. 
	
	On the other hand, the fact that $\{u_n\}_{n \in \N} \subset X$, implies that this sequence is uniformly bounded in $L^2(\W)$, and then, we can extract a weakly convergent subsequence $\{u_{n_j}\}_{j \in \N}$. Let $u \in L^2(\W)$ such that $u_{n_j} \rightharpoonup u$, our goal now is to show that $u_{n_j} \to u$ strongly in $L^{2}(\W)$, which implies strong convergence in $L^1(\W)$  since $|\W|<\infty$.  
	
	To this end, we only need to verify that $\|u_{n_j}\|_{L^2(\W)} \to \|u\|_{L^2(\W)}$ or, equivalently, $\|\Fu(u_{n_j})\|_{L^2(\rn)} \to \|\Fu(u)\|_{L^2(\rn)}$. From \eqref{eq:concentracion}, and the fact that $u \in L^2(\W)$, we can take $R$ large enough, in such a way that $\int_{|\xi|> R} \Fu^2(u_{n_j})[\xi] \,d\xi < \eta, \quad \forall j \in \N,$
	and
	$\int_{|\xi|> R} \Fu^2(u)[\xi] \,d\xi < \eta.$ 
	Then we have 
	\begin{equation}
		\label{eq:convergencia}
		\Big| \|\Fu(u_{n_j})\|^2_{L^2(\rn)} - \|\Fu(u)\|^2_{L^2(\rn)} \Big| \leq  
	\end{equation}
	$$\Big| \int_{|\xi| \leq R} \Fu^2(u_{n_j})[\xi] \, d\xi- \int_{|\xi| \leq R} \Fu^2(u)[\xi]\, d\xi \Big| + 2\eta.$$   
	Since $u_{n_j}$ is supported in $\W$, for all $j \in \N$, weak convergence implies $\Fu(u_{n_j})[\xi] \to \Fu(u)[\xi]$ for all $\xi \in \rn$. And, as we have observed before, since $\|u_{n_j}\|_{L^1(\W)} \leq C_0$ for all $j \in \N$, $\Fu^2(u_{n_j})[\xi] \leq C^2_0$ for all $\xi \in \rn$. Hence, we can apply Dominated Convergence Theorem in \eqref{eq:convergencia} and say that there exists $j_0$ in such a way that if $j>j_0$ then               
	
	\begin{equation}
		\label{eq:convergencia2}
		\Big| \|\Fu(u_{n_j})\|^2_{L^2(\rn)} - \|\Fu(u)\|^2_{L^2(\rn)} \Big| \leq 3\eta. 
	\end{equation}
	Since $\eta$ can be arbitrary small, we have $\|u_{n_j}\|_{L^2(\W)} \to \|u\|_{L^2(\W)}$, and the statement of the theorem follows.

\end{proof} 

\section{Numerical experiments} \label{sec:numerical}

In this section, three numerical examples are presented in order to explore the behavior of the solution under fractional parameters $s$ and $\alpha$.    

For the first example, we have used $\W = [-1,1]$, a uniform mesh consisting of $3000$ nodes, $s = 0.005$, $\alpha = 1$, $\eps^2=0.5$, and the function $v = -0.5 I_{(-1,0)} +  0.5 I_{[0,1)}$ as initial data. Here, the aim is to obtain some experimental support for the ideas displayed in section \ref{sec:asymptotic}, that is, the behavior with a small parameter $s$. Numerical results are summarized in Figure \ref{ejemplo_1}, and equilibrium values far from $1$ and $-1$ can be observed. Furthermore, equilibrium values seem to be placed near the values predicted in section \ref{sec:asymptotic} or, in other words, the solution seems to approximate a minimizer of \eqref{eq:funcional}.

\begin{figure}[h]
	\includegraphics[width=65mm]{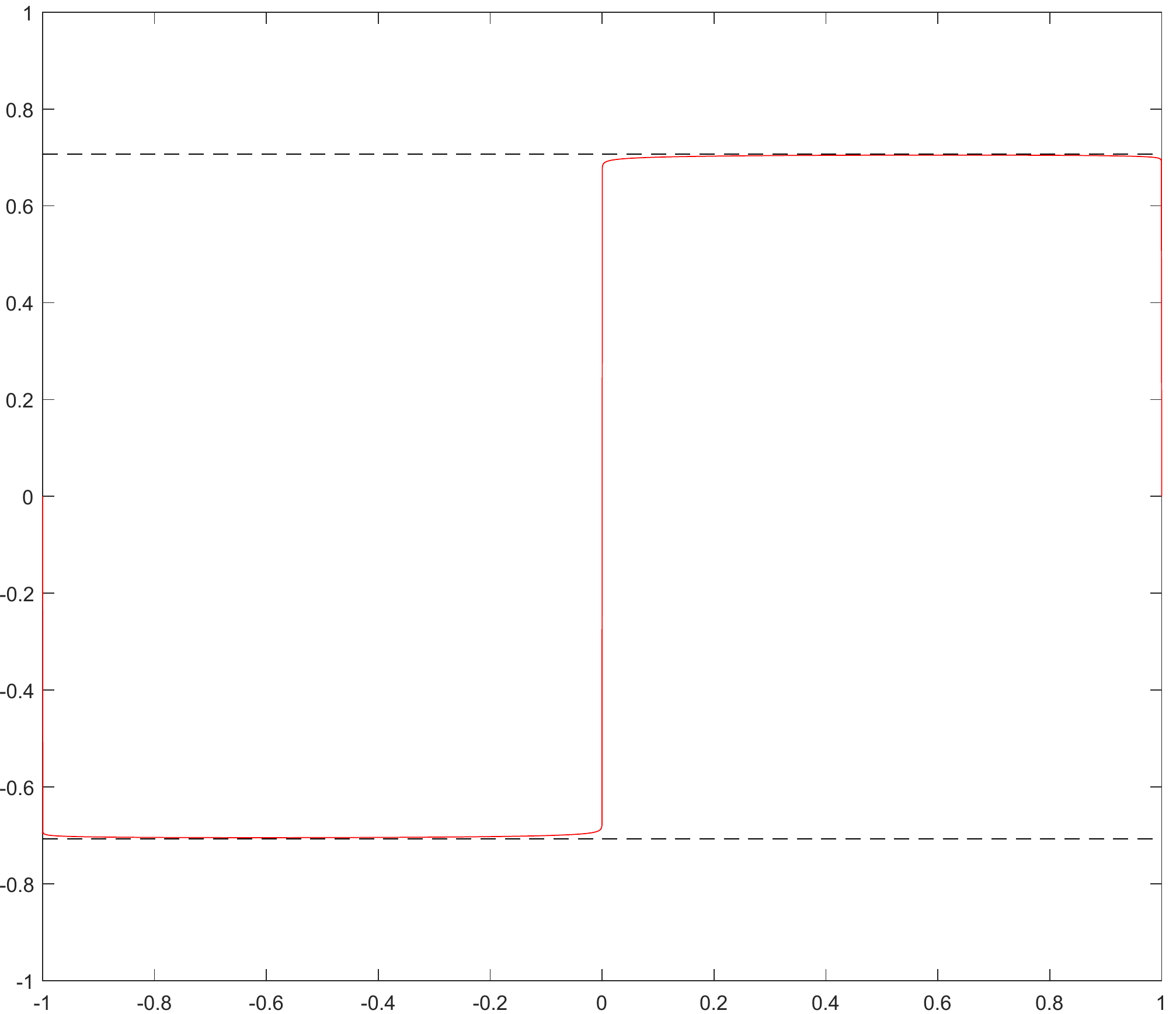}
	\caption{In red the solution of example 1 at $t = 50$, in black-dashed the values $\pm\sqrt{1-\eps^2}$. It can be seen that the equilibrium states remain far from $1$ and $-1$, unlike the behavior in the classic AC equation, and approach the values predicted in section \ref{sec:asymptotic} (see \eqref{eq:funcional}).} \label{ejemplo_1}
\end{figure}

Example 2 and 3 (spinodal decomposition) are shown in Figure \ref{ejemplo_s} and \ref{ejemplo_alfa} respectively. Here we have used $\W$ as the unitary ball, a uniform triangulation consisting of $16554$ triangles, $\eps^2 = 0.02$, and random noise as initial data. In example 2 (Figure \ref{ejemplo_s}), the parameter $\alpha$ is fixed in $1$, and results for several values of $s$ are shown. Can be observed the fact that, as we have mentioned in section \ref{sec:asymptotic},  the smaller the parameter $s$, the thinner the interface. Finally, example 3 (Figure \ref{ejemplo_alfa}) shows the behavior for fractional values of the parameter $\alpha$, with $s = 1$.        

\begin{figure}[h]
	\centering
	\begin{tabular}{|c|c|c|c|}
		\hline
		\subf{\includegraphics[width=22mm]{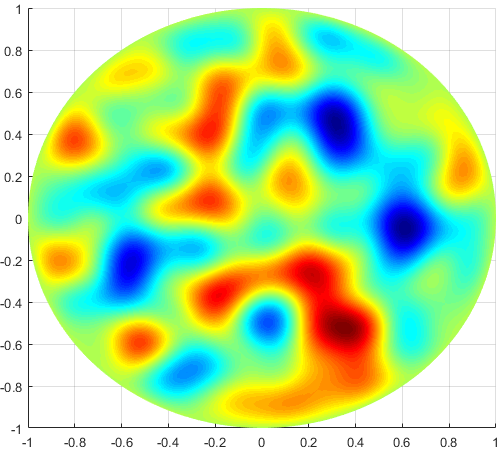}}
		{$s = 1$, $t = 2.5$}
		&
		\subf{\includegraphics[width=22mm]{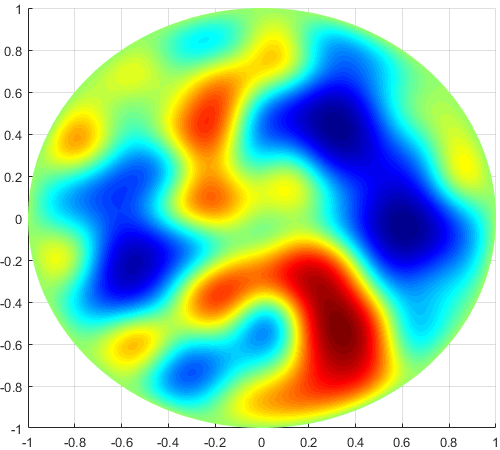}}
		{$s = 1$, $t = 5$}
		&
		\subf{\includegraphics[width=22mm]{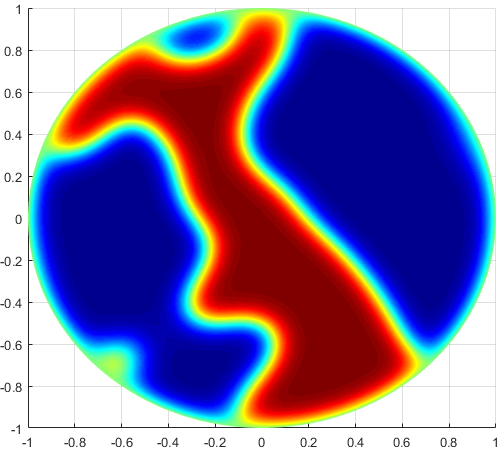}}
		{$s = 1$, $t = 10$}
		\\
		\hline
		\subf{\includegraphics[width=22mm]{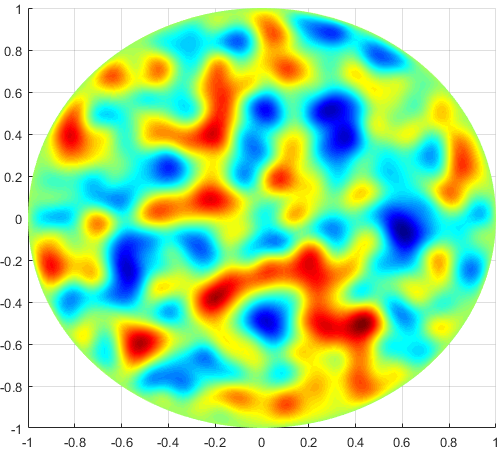}}
		{$s = 0.85$, $t = 2.5$}
		&
		\subf{\includegraphics[width=22mm]{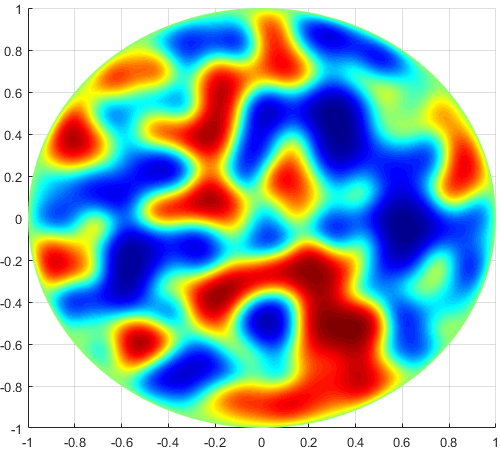}}
		{$s = 0.85$, $t = 5$}
		&
		\subf{\includegraphics[width=22mm]{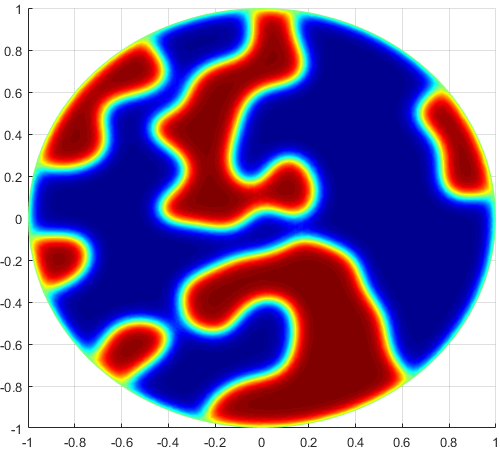}}
		{$s = 0.85$, $t = 10$}
		\\
		\hline
		\subf{\includegraphics[width=22mm]{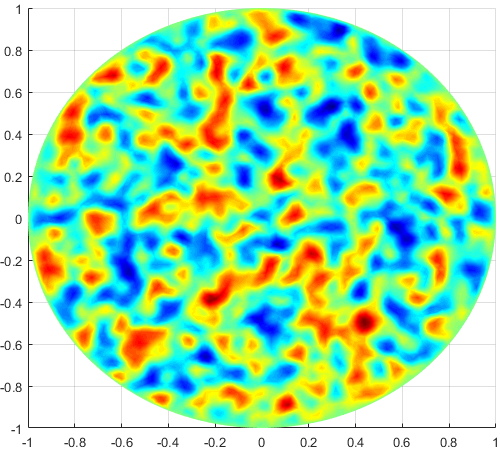}}
		{$s = 0.7$, $t = 2.5$}
		&
		\subf{\includegraphics[width=22mm]{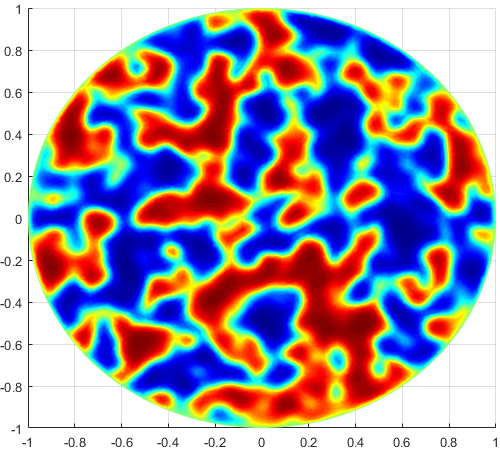}}
		{$s = 0.7$, $t = 5$}
		&
		\subf{\includegraphics[width=22mm]{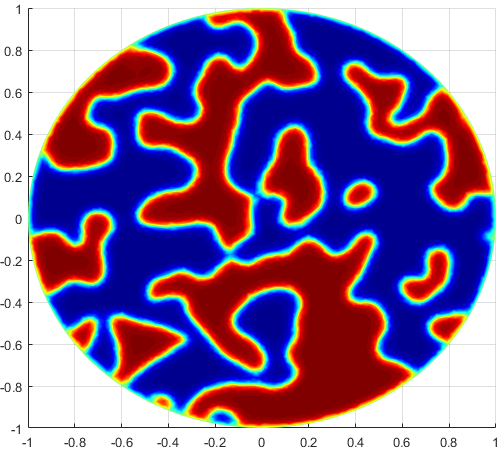}}
		{$s = 0.7$, $t = 10$}
		\\
		\hline
	\end{tabular}
	\caption{In this example we set $\Omega = B(0,1)$, $\alpha = 1$, and random noise as initial condition. The evolution in time is displayed for several values of $s$.}
	\label{ejemplo_s}
\end{figure}

\begin{figure}[h]
	\centering
	\begin{tabular}{|c|c|c|c|}
		\hline
		\subf{\includegraphics[width=22mm]{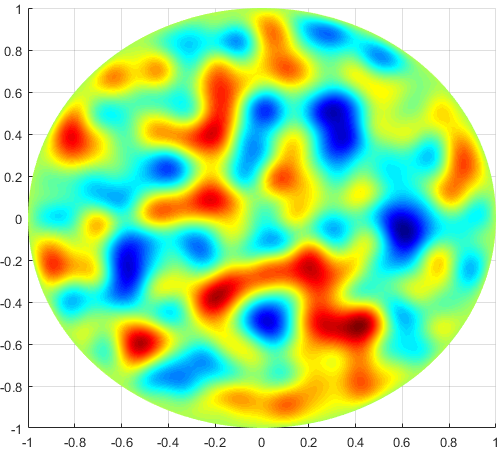}}
		{$\alpha = 1$, $t = 1.25$}
		&
		\subf{\includegraphics[width=22mm]{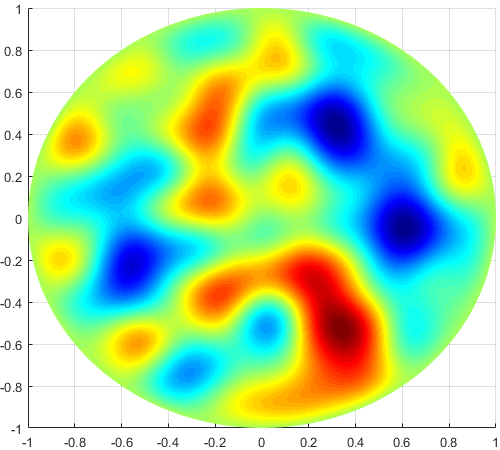}}
		{$\alpha = 1$, $t = 3.75$}
		&
		\subf{\includegraphics[width=22mm]{s1_200.png}}
		{$\alpha = 1$, $t = 10$}
		\\
		\hline
		\subf{\includegraphics[width=22mm]{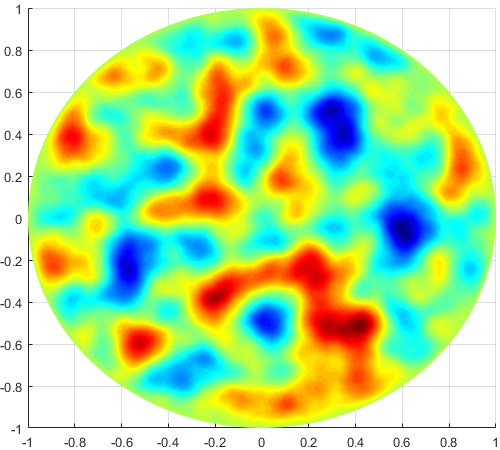}}
		{$\alpha = 0.7$, $t = 1.25$}
		&
		\subf{\includegraphics[width=22mm]{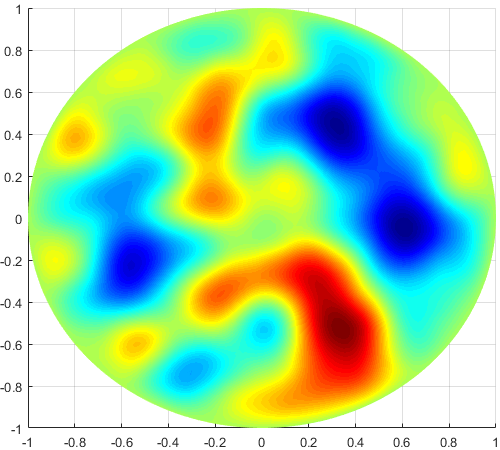}}
		{$\alpha = 0.7$, $t = 3.25$}
		&
		\subf{\includegraphics[width=22mm]{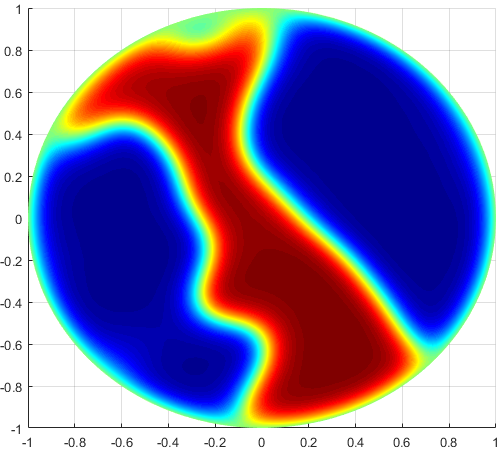}}
		{$\alpha = 0.7$, $t = 10$}
		\\
		\hline
		\subf{\includegraphics[width=22mm]{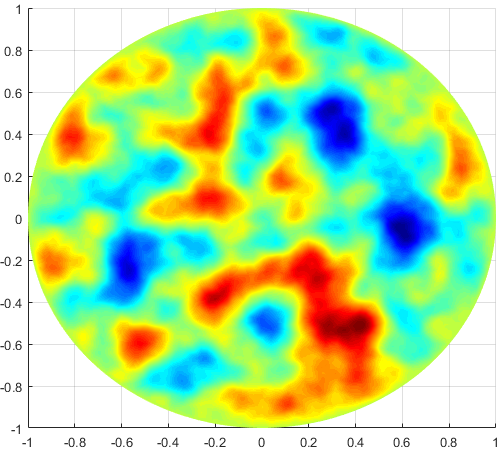}}
		{$\alpha = 0.4$, $t = 1.25$}
		&
		\subf{\includegraphics[width=22mm]{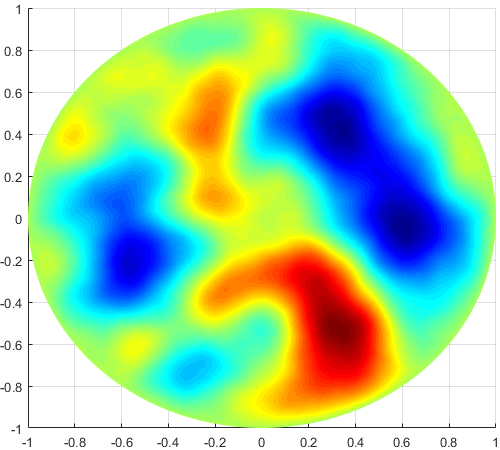}}
		{$\alpha = 0.4$, $t = 3.25$}
		&
		\subf{\includegraphics[width=22mm]{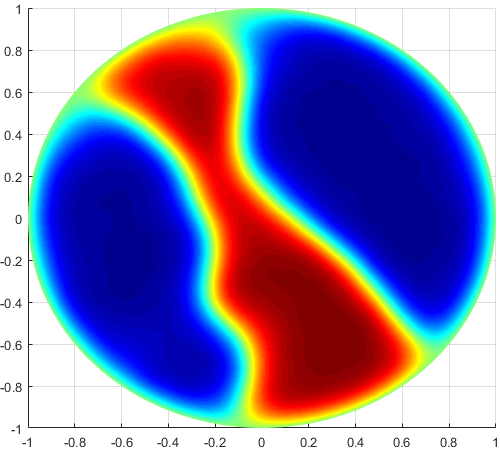}}
		{$\alpha = 0.4$, $t = 10$}
		\\
		\hline
	\end{tabular}
	\caption{In this example we set $\Omega = B(0,1)$, $s = 1$, and random noise as initial condition. The evolution in time is displayed for several values of $\alpha$.}
	\label{ejemplo_alfa}
\end{figure}

\appendix 

\section{Auxiliary results}
\label{appendix_AC}
\begin{lemma}
	\label{lem:leinbiz}
	Let $f \in C([0,T],\ele)$, with $f$ differentiable in $(0,T)$ in such a way that $\| f'(t) \|_{\ele} \leq C t^{-\gamma}$ with $\gamma \in (0,1)$ for all $t \in (0,T)$. Then we have 
	\begin{equation}
		\label{eq:leibniz0}
		\partial_t \left( \int^t_0 f(t-s) \, ds \right)= f(0) + \int^t_0 \partial_t f(t-s) \,ds \quad \forall t \in (0,T).
	\end{equation}

\end{lemma}

\begin{proof}
	We write 
	\begin{equation}
		\label{eq:leibniz1}
		\int^{t+h}_0 f(t+h -s) \, ds - \int^{t}_0 f(t-s) \, ds
	\end{equation}
	$$= \int^{t}_0 f(t+h -s) - f(t-s) \, ds + \int^{t+h}_t f(t+h -s) \, ds.$$
	From the mean value inequality in Banach spaces (see, for instance, \cite[Appendix B]{tesismikkola}) we can estimate
	
	$$\|f(t+h-s) - f(t-s)\|_{\ele} \leq  h\|f'(r-s)\|_{\ele}$$
	$$\leq Ch(r-s)^{-\gamma} \leq Ch(t-s)^{-\gamma}, \quad \forall s \in [0,t).$$
	This, along with the fact that $f'$ exists in $(0,T)$, allows us to use the Domitated Convergence Theorem (see \cite[Appendix B]{tesismikkola}) and get 
	\begin{equation}
		\label{eq:leibniz2}
		\lim_{h \to 0} \int^{t}_0 \frac{f(t+h -s) - f(t-s)}{h} \, ds = \int^{t}_0 f'(t -s) \, ds.
	\end{equation}

	On the other hand, from the fact that $f(t)$ is a continuous function in $t = 0$, we have
	
	\begin{equation}
		\label{eq:leibniz3}
		\frac{1}{h} \int^{t+h}_{t}f(t+h-s) \, ds = \frac{1}{h} \int^{0}_{h}-f(r) \, dr = \frac{1}{h} \int^{h}_{0}f(r) \, dr \xrightarrow[h \to 0]{} f(0),
	\end{equation}
	
	where the convergence is in $\ele$ sense.
	
	Finally, combining \eqref{eq:leibniz1}, \eqref{eq:leibniz2} and \eqref{eq:leibniz3}, we obtain \eqref{eq:leibniz0}.

\end{proof}

\begin{lemma}
	\label{lem:gronwall_discreto}
	For $\tau>0$ consider $t_n = n\tau$ with $n = 1,...,N$ and $T=N\tau$. If
	
	\begin{equation}
		\label{eq:gronwall_dis1}
		\varphi_n \leq At_n^{-1+\alpha} + B\tau \sum^{n-1}_{j=1} t^{-1+\beta}_{n-j} \varphi_j, 
	\end{equation}
	for all $n = 1,...,N$, with some constants $A$, $B \geq 0$, and $\alpha$, $\beta > 0$, then there exist a constant $C = C(B,T,\alpha,\beta)$ such that
	
	\begin{equation}
		\label{eq:gronwall_dis2}
		\varphi_n \leq CAt^{-1+\alpha}_n, \qquad n = 1,...,N.
	\end{equation}	
	
\end{lemma}

\begin{proof}
	First, we observe the following relation. For $\gamma$, $\delta > 0$ we have
	
	\begin{equation}
		\label{eq:gronwall_dis3}
		\tau \sum_{j=0}^{n-1} t^{-1+\gamma}_{n-j} t^{-1+\delta}_{j+1} \leq \sum_{j=0}^{n-1} \int_{t_j}^{t_{j+1}}(t_n - s)^{-1+\gamma}s^{-1+\delta} \,ds 
	\end{equation}
	
	$$ \leq \int_{0}^{t_n}(t_n - s)^{-1+\gamma}s^{-1+\delta} \,ds = B(\gamma,\delta)t^{-1+\gamma + \delta }_n,$$
	where $B$ denotes the beta function. 
	
	Now, if we choose the smallest $k=k(\beta) \in \mathbb{N}$ such that $-1+k\beta>0$ and iterate \eqref{eq:gronwall_dis1} $k-1$ times, using relation \eqref{eq:gronwall_dis2} we obtain
	
	\begin{equation}
		\label{eq:gronwall_dis4}
		\varphi_n \leq D_1 A t^{-1+\alpha} + D_2 \tau \sum_{j=1}^{n-1} t^{-1+k\beta}_{n-j}\varphi_j \leq D_1 A t^{-1+\alpha} + D_2   T^{-1+k\beta}\tau\sum_{j=1}^{n-1}\varphi_j,
	\end{equation}  
	with $D_1 = D_1(C_2,T,\beta)$ and $D_2 = D_2(C_2,\beta)$. If $-1 + \alpha \geq 0$ we can derive \eqref{eq:gronwall_dis2} by means of a standard discrete Gronwall type inequality. Otherwise, we take $\psi_n = t_n^{1-\alpha} \varphi_n$ and using \eqref{eq:gronwall_dis4} we obtain 
	
	\begin{equation*}
		\psi_n \leq D_1 A + D_2 T^{-1+k\beta}\tau\sum_{j=1}^{n-1} t^{-1+\alpha}_j\psi_j,
	\end{equation*}  
	and deduce $t_n^{1-\alpha} \varphi_n \leq D_3 A$, again by means of a standard Gronwall type inequality.  
	
\end{proof}

\begin{proof}\textit{(of Lemma \ref{lem:lemma_seq})}
	From the definition of $\{\wtilde_n\}_{n \in \mathbb{N}_0}$, we know that
	$(1 - \xi)^{\alpha} = \sum^{\infty}_{j=0} \wtilde_j \xi^j. $
	Then, defining $g(\xi) = 1 - (1 - \xi)^{\alpha}$, and recalling that $w_0 = 1$,  we have $g(\xi) = \sum^{\infty}_{j=1} -\wtilde_j \xi^j. $
	Now, defining $f(\xi) = \sum^{\infty}_{j=0} c_j \xi^j$, from the definition of $\{c_n\}_{n \in \mathbb{N}_0}$, and using the Cauchy product for power series, the following equality can be easily checked, 
	
	\begin{equation}
		\label{eq:equal_app}
		f(\xi) \frac{g(\xi)}{\xi} = \frac{f(\xi) - c_0}{\xi}.
	\end{equation} 
	Recalling that $c_0 = 1$, and $-\wtilde_1 = \alpha$, from \eqref{eq:equal_app} we can obtain an explicit expression for $f$, $f(\xi) = (1-\xi)^{-\alpha}.$ 
	It is well known that series expansion of $f$ is 
	$f(\xi) =  \sum^{\infty}_{j=0} (-1)^{j} {-\alpha \choose j } \xi^j.$
	Then $c_n = (-1)^{n} {-\alpha \choose n },$ 
	where ${-\alpha \choose n } = \frac{\Gamma(1-\alpha)}{\Gamma(1 + n)\Gamma(1 - n -\alpha)}.$

	Finally, by means of basic Gamma function properties, we can verify that ${-\alpha \choose n } \in O(n^{\alpha-1})$, and hence, $\{c_n\}_{n \in \mathbb{N}_0} \in O(n^{\alpha-1})$.  \end{proof}

\subsection*{Acknowledgments} The authors thank Prof. Juli\'an Fern\'andez Bonder and Prof. Ciprian Gal for their valuable comments which helped to improve the manuscript.

%
%

\bibliography{AC_tex}{}

\begin{thebibliography}{10}

\bibitem{ABB}
G.~Acosta, F.~Bersetche, and J.P. Borthagaray.
\newblock A short {FEM} implementation for a 2d homogeneous {D}irichlet problem
  of a fractional {L}aplacian.
\newblock {\em Comp. Math. Appl.}, 74(4):784--816, 2017.

\bibitem{ABB2}
G.~Acosta, F.~M. Bersetche, and J.~P. Borthagaray.
\newblock Finite element approximations for fractional evolution problems.
\newblock {\em Fractional Calculus and Applied Analysis}, 22(3):767--794, 2019.

\bibitem{AcostaBorthagaray}
G.~Acosta and J.~P. Borthagaray.
\newblock A fractional {L}aplace equation: Regularity of solutions and finite
  element approximations.
\newblock {\em SIAM J. Numer. Anal.}, 55(2):472--495, 2017.

\bibitem{glusa}
M.~Ainsworth and C.~Glusa.
\newblock Aspects of an adaptive finite element method for the fractional
  laplacian: a priori and a posteriori error estimates, efficient
  implementation and multigrid solver.
\newblock {\em Computer Methods in Applied Mechanics and Engineering},
  327:4--35, 2017.

\bibitem{ains}
M.~Ainsworth and Z.~Mao.
\newblock Analysis and approximation of a fractional {C}ahn--{H}illiard
  equation.
\newblock {\em SIAM Journal on Numerical Analysis}, 55(4):1689--1718, 2017.

\bibitem{ains2}
M.~Ainsworth and Z.~Mao.
\newblock Well-posedness of the {C}ahn--{H}illiard equation with fractional
  free energy and its fourier galerkin approximation.
\newblock {\em Chaos, Solitons \& Fractals}, 102:264--273, 2017.

\bibitem{akagi}
G.~Akagi, G.~Schimperna, and A.~Segatti.
\newblock Fractional {C}ahn--{H}illiard, {A}llen--{C}ahn and porous medium
  equations.
\newblock {\em Journal of Differential Equations}, 261(6):2935--2985, 2016.

\bibitem{AC_79}
S.~M. Allen and J.~W. Cahn.
\newblock A microscopic theory for antiphase boundary motion and its
  application to antiphase domain coarsening.
\newblock {\em Acta Metallurgica}, 27(6):1085--1095, 1979.

\bibitem{tesis_yo}
F.~M. Bersetche.
\newblock {\em Numerical methods for non-local evolution problems}.
\newblock PhD thesis, Universidad de Buenos Aires, 2019.

\bibitem{BdPM}
Juan~Pablo Borthagaray, Leandro~M Del~Pezzo, and Sandra Mart{\'\i}nez.
\newblock Finite element approximation for the fractional eigenvalue problem.
\newblock {\em Journal of Scientific Computing}, 77(1):308--329, 2018.

\bibitem{gamma}
A.~Braides.
\newblock {\em Gamma-convergence for Beginners}, volume~22.
\newblock Clarendon Press, 2002.

\bibitem{Brezis}
Haim Brezis.
\newblock {\em Functional analysis, {S}obolev spaces and partial differential
  equations}.
\newblock Universitext. Springer, New York, 2011.

\bibitem{CH_58}
J.~W. Cahn and J.~E. Hilliard.
\newblock Free energy of a nonuniform system. i. interfacial free energy.
\newblock {\em The Journal of chemical physics}, 28(2):258--267, 1958.

\bibitem{tesisneto}
P.~Mendes de~Carvalho~Neto.
\newblock {\em Fractional differential equations: a novel study of local and
  global solutions in Banach spaces}.
\newblock PhD thesis, ICMC-USP, 2013.

\bibitem{Hitchhikers}
Eleonora Di~Nezza, Giampiero Palatucci, and Enrico Valdinoci.
\newblock Hitchhiker's guide to the fractional {S}obolev spaces.
\newblock {\em Bull. Sci. Math.}, 136(5):521--573, 2012.

\bibitem{Diethelm}
K.~Diethelm.
\newblock {\em The analysis of fractional differential equations}, volume 2004
  of {\em Lecture Notes in Mathematics}.
\newblock Springer-Verlag, Berlin, 2010.
\newblock An application-oriented exposition using differential operators of
  Caputo type.

\bibitem{elliott}
C.~M. Elliott and S.~Larsson.
\newblock Error estimates with smooth and nonsmooth data for a finite element
  method for the {C}ahn-{H}illiard equation.
\newblock {\em Mathematics of Computation}, 58(198):603--630, 1992.

\bibitem{FR}
X.~Fern{\'a}ndez-Real and X.~Ros-Oton.
\newblock Boundary regularity for the fractional heat equation.
\newblock {\em Revista de la Real Academia de Ciencias Exactas, Fisicas y
  Naturales. Serie A. Matematicas}, 110(1):49--64, 2016.

\bibitem{gal17}
C.~Gal and M.~Warma.
\newblock Fractional in time semilinear parabolic equations and applications.
\newblock {\em HAL Id: hal-01578788}, 2017.

\bibitem{grubb_autovalores}
G.~Grubb.
\newblock Spectral results for mixed problems and fractional elliptic
  operators.
\newblock {\em J. Math. Anal. Appl.}, 421(2):1616--1634, 2015.

\bibitem{HPH}
D.~He, K.~Pan, and H.~Hu.
\newblock A fourth-order maximum principle preserving operator splitting scheme
  for three-dimensional fractional {A}llen-{C}ahn equations.
\newblock {\em arXiv preprint arXiv:1804.07246}, 2018.

\bibitem{HTY}
T.~Hou, T.~Tang, and J.~Yang.
\newblock Numerical analysis of fully discretized crank--nicolson scheme for
  fractional-in-space {A}llen--{C}ahn equations.
\newblock {\em Journal of Scientific Computing}, 72(3):1214--1231, 2017.

\bibitem{BLPZ}
B.~Jin, R.~Lazarov, J.~Pasciak, and Z.~Zhou.
\newblock Error analysis of semidiscrete finite element methods for
  inhomogeneous time-fractional diffusion.
\newblock {\em IMA Journal of Numerical Analysis}, 35(2):561--582, 2014.

\bibitem{Jin}
B.~Jin, R.~Lazarov, and Z.~Zhou.
\newblock Two fully discrete schemes for fractional diffusion and
  diffusion-wave equations with nonsmooth data.
\newblock {\em SIAM J. Sci. Comput.}, 38(1):A146--A170, 2016.

\bibitem{KM}
Michael Karkulik and Jens~Markus Melenk.
\newblock {H} -matrix approximability of inverses of discretizations of the
  fractional {L}aplacian.
\newblock {\em Advances in Computational Mathematics}, 45(5-6):2893--2919,
  2019.

\bibitem{larsson}
S.~Larsson.
\newblock Semilinear parabolic partial differential equations: theory,
  approximation, and application.
\newblock {\em New trends in the mathematical and computer sciences},
  3:153--194, 2006.

\bibitem{LWY}
Z.~Li, H.~Wang, and D.~Yang.
\newblock A space--time fractional phase-field model with tunable sharpness and
  decay behavior and its efficient numerical simulation.
\newblock {\em Journal of Computational Physics}, 347:20--38, 2017.

\bibitem{hongwang}
H.~Liu, A.~Cheng, H.~Wang, and J.~Zhao.
\newblock Time-fractional {A}llen--{C}ahn and {C}ahn--{H}illiard phase-field
  models and their numerical investigation.
\newblock {\em Computers \& Mathematics with Applications}, 76(8):1876--1892,
  2018.

\bibitem{Lub2}
C.~Lubich.
\newblock Convolution quadrature and discretized operational calculus. {I}.
\newblock {\em Numer. Math.}, 52(2):129--145, 1988.

\bibitem{Lub1}
C.~Lubich.
\newblock Convolution quadrature revisited.
\newblock {\em BIT}, 44(3):503--514, 2004.

\bibitem{tesismikkola}
K.~Mikkola.
\newblock {\em Infinite-Dimensional Linear Systems, Optimal Control and
  Algebraic Riccati Equations}.
\newblock PhD thesis, Helsinki University of Technology Institute of
  Mathematics, 2002.

\bibitem{Podlubny}
I.~Podlubny.
\newblock {\em Fractional differential equations}, volume 198 of {\em
  Mathematics in Science and Engineering}.
\newblock Academic Press, Inc., San Diego, CA, 1999.
\newblock An introduction to fractional derivatives, fractional differential
  equations, to methods of their solution and some of their applications.

\bibitem{RosOtonSerra2}
X.~Ros-Oton and J.~Serra.
\newblock Local integration by parts and {P}ohozaev identities for higher order
  fractional {L}aplacians.
\newblock {\em Discrete Contin. Dyn. Syst.}, 35(5):2131--2150, 2015.

\bibitem{SV}
O.~Savin and E.~Valdinoci.
\newblock $\gamma$-convergence for nonlocal phase transitions.
\newblock In {\em Annales de l'Institut Henri Poincare (C) Non Linear
  Analysis}, volume~29, pages 479--500. Elsevier Masson, 2012.

\bibitem{SXKE}
F.~Song, C.~Xu, and G.~E. Karniadakis.
\newblock A fractional phase-field model for two-phase flows with tunable
  sharpness: Algorithms and simulations.
\newblock {\em Computer Methods in Applied Mechanics and Engineering},
  305:376--404, 2016.

\bibitem{stynes}
M.~Stynes.
\newblock Too much regularity may force too much uniqueness.
\newblock {\em Fractional Calculus and Applied Analysis}, 19(6):1554--1562,
  2016.

\end{thebibliography}
\bibliographystyle{plain}
\end{document}